\documentclass[11pt]{amsart}

\usepackage{amsmath,amssymb,amsthm,mathtools,mathrsfs}
\usepackage[a4paper,margin=2.7cm]{geometry}
\usepackage[colorlinks=true,linkcolor=blue,citecolor=blue,urlcolor=blue]{hyperref}
\usepackage{listings}
\usepackage{verbatim} 
\usepackage{longtable,array,booktabs,needspace}

\newtheorem{theorem}{Theorem}[section]
\newtheorem{proposition}[theorem]{Proposition}
\newtheorem{lemma}[theorem]{Lemma}
\newtheorem{corollary}[theorem]{Corollary}
\theoremstyle{definition}

\theoremstyle{remark}
\newtheorem{remark}[theorem]{Remark}
\theoremstyle{definition}
\newtheorem{thmletter}{Theorem} 

\theoremstyle{remark}

\newcommand{\p}{\partial}
\newcommand{\m}{\mathfrak m}
\newcommand{\ord}{\operatorname{ord}}
\newcommand{\Sing}{\operatorname{Sing}}
\newcommand{\Spec}{\operatorname{Spec}}

\title[ Johnson-Kollár Open problem]{On Two Open Problems by Johnson and Kollár}

\author{Siyuan Fu}
\address{Weiyang College, Tsinghua University}
\email{fusy24@mails.tsinghua.edu.cn}

\author{Zhiming Li}
\address{Tanwei College, Tsinghua University}
\email{zm-li25@mails.tsinghua.edu.cn}

\author{Huaiqing Zuo}
\address{Department of Mathematical Sciences,
	Tsinghua University,
	Beijing, 100084, P. R. China.}
\email{hqzuo@mail.tsinghua.edu.cn}

\begin{document}

\begin{abstract}
We resolve two problems left open by Johnson and Koll\'ar in \cite{JK}: Question~21 and the openness issue arising in Example~14.

Keywords.  Nash problem, arc space, jet schemes, isolated compound Du Val singularities.
		
		MSC(2020).  Primary 14E18; Secondary 32S05.
\end{abstract}

\maketitle
\section{Introduction}
The Nash problem, introduced by John Nash in a manuscript written in 1968 and published in 1995 \cite{Na}, concerns the structure of the arc space of a singular algebraic variety. More precisely, the Nash problem asks whether the Nash map from the irreducible components of the space of arcs centered in the singular locus of a variety \(X\) to the essential divisors over \(X\) is bijective. This problem bridges the gap between the local geometry of singularities and the global, infinite-dimensional structure of arc spaces—a connection that has far-reaching implications in areas such as motivic integration, birational geometry, and the theory of singularities. The paper \cite{JK} by Jennifer M. Johnson and János Kollár offers an exposition of the Nash problem.

A central family of examples in \cite{JK} consists of the surface singularities
\[
x^{2}+y^{3}=z^{n}.
\]
For several values of \(n\), Johnson and Kollár describe the irreducible arc families centered at the origin by analyzing the orders of vanishing and constructing explicit parameterizations. Their computations also exhibit cases in which the number of irreducible arc families remains unchanged under the addition of higher-order terms to the defining equation. These examples motivate the study of how such decompositions behave under perturbations and in higher dimensions. This leads to the following question posed in \cite[Question~21]{JK}.

% Johnson and Kollár carefully guide the reader through the Nash problem for a prominent family of surface singularities defined by\[x^{2} + y^{3} = z^{n},\]for various positive integers \(n\). Through a succession of worked examples (\(n=2,3,4,5,6,\dots\)), they illustrate the often delicate and case‑dependent analysis required to decompose the arc space into finitely many irreducible families, each admitting a free parametrisation. These examples, while elementary, capture the essential difficulties of the general theory: the interplay of orders of vanishing, the role of coordinate changes, and the ubiquitous need for Newton's method of rotating rulers. Most importantly, they show that the number of irreducible arc families is stable under certain perturbations (e.g., adding higher‑order terms to the defining equation). However, the authors explicitly warn that this stability is ``unlikely to be a general principle,'' and they conclude the paper by posing a concrete, open challenge \cite{JK}.

\vspace{0.5cm}

\textbf{Johnson--Kollár open problem A \cite[Question~21]{JK}.} For $N \in \{3,4,5\}$ and an integer $M \ge N$, set
\[f_{N,\,M    }=x^2+y^3-z^N-w^M,\quad 
X_{N,\,M}
=
\{f_{N,\,M  }=0\}
\quad\]and let \(\operatorname{Arc}^{*}(X_{N,\,M})\) denote the arc space centered at the origin of \(X_{N,\,M}\). What are the irreducible components of \(\operatorname{Arc}^{*}(X_{N,\,M})\) for \(N = 3,\,4,\,5\) and every \(M\ge N\)?

This question extends the surface analysis of Johnson and Koll\'ar from \(\mathbb{C}^{3}\) to a family of three-dimensional hypersurfaces in \(\mathbb{C}^{4}\). The cases \(N=3,4,5\) arise from Du Val singularities of types \(D_4,E_6,E_8\), respectively \cite{Reid}.

In this paper, we answer  the Johnson--Kollár open problem A. For a scheme \(Y\), the notation \(Y_{\mathrm{red}}\) denotes the reduced scheme associated to \(Y\). Since a scheme \(Y\) and \(Y_{\mathrm{red}}\) have canonically homeomorphic
underlying topological spaces, passing to the reduction does not change irreducible components or dimension. We therefore work with the reduced scheme \(\operatorname{Arc}^{0}(X_{N,\,M}):=\operatorname{Arc}^{*}(X_{N,\,M})_{\mathrm{red}}\). Theorem \ref{mta} records the number of irreducible components of \(\operatorname{Arc}^{0}(X_{N,\,M})\) in different cases.

\begin{thmletter}\label{mta}
For every integer
\[
N\in\{3,4,5\},
\qquad
M\geq N,
\]
the number of irreducible components is given by
the following table:
\begin{longtable}{>{$}c<{$}|>{$}c<{$}|>{$}c<{$}}
N&M&\#\operatorname{Irr}(\operatorname{Arc}^{0}(X_{N,\,M}))\\ \hline
\endfirsthead
N&M&\#\operatorname{Irr}(\operatorname{Arc}^{0}(X_{N,\,M}))\\ \hline
\endhead
3&3&1\\
3&4\le M\le5&3\\
3&M\ge6&4\\ \hline
4&4&1\\
4&5&2\\
4&6\le M\le7&3\\
4&8&4\\
4&9\le M\le11&5\\
4&M\ge12&6\\ \hline
5&5&5\\
5&6\le M\le7&1\\
5&8&3\\
5&9&4\\
5&10\le M\le11&2\\
5&12\le M\le13&3\\
5&14&5\\
5&15\le M\le17&4\\
5&18\le M\le19&5\\
5&20\le M\le23&6\\
5&24\le M\le29&7\\
5&M\ge30&8
\end{longtable}
\end{thmletter}

The notation for exact-order strata and the jet-stratification
argument for Theorem~\ref{mta} are developed in
Section~\ref{sec:solution}.

\textbf{Johnson--Kollár open problem B.} The second problem arises from Johnson--Koll\'ar's discussion of the
coefficientwise product topology and their computation of the arc
families on
\[
X_5=\{x^2+y^3=z^5\}\subset\mathbb C^3
\]
\cite[Paragraph~5, p.~522 and Example~14, pp.~530--531]{JK}.
Let $\mathcal A$ denote the centered arc scheme
$\operatorname{Arc}^0(X_5)$, whose precise scheme-theoretic
realization is fixed in Subsection~\ref{subsec:e8-setup}. Let
$W_1,\ldots,W_8$ and $C_1,\ldots,C_8$ be the exact-order families and the
maximal divisorial sets defined there. The problem is to determine whether
the eight sets $W_i(\mathbb C)$ are open in the entire space
$\mathcal A(\mathbb C)$ for the coefficientwise product topology. Theorem \ref{mtb} below answers this question.

\begin{thmletter}\label{mtb}
Each of the eight families $W_i(\mathbb C)$ is open in
$\mathcal A(\mathbb C)$ for the coefficientwise product topology.
More precisely, for $i\ne j$ one has
\[
W_i(\mathbb C)\cap C_j(\mathbb C)=\emptyset,
\qquad
W_i(\mathbb C)\cap
\overline{C_j(\mathbb C)}^{\,\mathrm{prod}}=\emptyset.
\]
Moreover, the eight families described in Johnson--Koll\'ar's
Example~14 are precisely the families
\[
W_1(\mathbb C),\ldots,W_8(\mathbb C)
\]
in the order specified in Subsection~\ref{subsec:e8-setup}.
\end{thmletter}

The distinction between relative openness in a higher-contact locus
and openness in the ambient arc space is made explicit in
Subsections~\ref{subsec:e8-setup} and
\ref{subsec:e8-product-openness}. Section~\ref{sec:preliminaries}
fixes the common notation for arcs, jets, and orders of vanishing.
The definitions and arguments specific to Theorem~\ref{mtb} are
given in Section~\ref{sec:e8-open-families}.

\section{Preliminaries and notation}
\label{sec:preliminaries}

Throughout the paper, the ground field is \(\mathbb C\). We write \(\mathbb A^r=\mathbb A^r_{\mathbb C}\) and \(\mathbb C^\times=\mathbb C\setminus\{0\}\).
For a scheme \(Y\), the notation \(Y_{\mathrm{red}}\) denotes its reduced scheme, \(\operatorname{Sing}Y\) its singular locus, and \(\operatorname{Irr}(Y)\) the set of irreducible components of its underlying Zariski topological space. The symbol \(\#\operatorname{Irr}(Y)\) denotes the cardinality of this set. Passing to the reduction does not change irreducible components
or dimension. Unless another topology is explicitly specified, closures and irreducibility are understood in the Zariski topology; \(\overline{S}^{\,Y}\) denotes the closure of \(S\) in \(Y\).

For an ideal \(I\) of a ring \(R\), \(V(I)\subset \operatorname{Spec}R\) is its zero locus. We also use \(Z(g_1,\ldots,g_s)\) for \(V((g_1,\ldots,g_s))\), and
\[
D(g)=\operatorname{Spec}R\setminus V(g)
     =\operatorname{Spec}R_g
\]
for a principal open subset. Thus, \(R_g\) means localization at the powers of \(g\); \(R^\times\) denotes the group of units of \(R\), and \(\operatorname{Min}R\) denotes the set of minimal prime ideals of \(R\). The integer-part notation \([u]\), when used in numerical estimates, means \(\lfloor u\rfloor\); \(\lceil u\rceil\) denotes the ceiling.

Finally, for a Noetherian local ring \((R,\mathfrak n)\) with
residue field \(\kappa\), we use
\[
\operatorname{edim}R=\dim_\kappa(\mathfrak n/\mathfrak n^2),
\qquad
\operatorname{ecodim}R=\operatorname{edim}R-\dim R.
\]
The abbreviation cDV stands for compound Du Val, and hcDV stands for
higher compound Du Val in the sense of \cite{dFW}.
The notation \(m\gg0\) means that \(m\) is sufficiently large
after the variety, and in particular \(N,M\), has been fixed.

\subsection{Formal arcs and jets}
\label{subsec:prelim-arcs-jets}

We use the standard conventions for arc spaces and jet schemes in \cite{EM}. 

For a complex variety \(X\) and a field extension \(L/\mathbb C\), an \(L\)-valued formal arc on \(X\) is a morphism
\[
\gamma:\operatorname{Spec}L[[t]]\longrightarrow X.
\]
Its center \(\gamma(0)\) is the image of the closed point \(t=0\).
The arc space is denoted by \(\operatorname{Arc}(X)\).

For \(m\in\mathbb Z_{\geq0}\), an \(L\)-valued \(m\)-jet is a morphism
\[
\beta:\operatorname{Spec}\bigl(L[t]/(t^{m+1})\bigr)\longrightarrow X,
\]
and \(J_m(X)\) denotes the scheme parametrizing such jets. The integer
\(m\) is the jet level; an \(m\)-jet retains precisely the coefficients
through the degree \(m\). We identify \(J_0(X)\) canonically with \(X\).

Evaluation at \(t=0\) and truncation modulo \(t^{m+1}\) are denoted by
\[
\begin{aligned}
\tau_{\infty,0}^X&:\operatorname{Arc}(X)\longrightarrow J_0(X)=X,
&\qquad
\tau_{m,0}^X&:J_m(X)\longrightarrow J_0(X)=X,\\
\tau_{\infty,m}^X&:\operatorname{Arc}(X)\longrightarrow J_m(X).
\end{aligned}
\]
In \(\tau_{a,b}^X\), the first subscript records the source level and
the second the target level; \(\infty\) denotes the arc level. Thus
\[
\tau_{\infty,0}^X
=\tau_{m,0}^X\circ\tau_{\infty,m}^X.
\]
For \(m\geq\ell\), write
\(\tau_{m,\ell}^X:J_m(X)\to J_\ell(X)\) for truncation.
In Section~\ref{sec:solution}, the notation \(p_m\) abbreviates
\(\tau_{\infty,m}^X\), with the variety understood from the context.
If \(0\in X\),
Johnson--Kollár's notation for the arc space centered at the origin is
\[
\operatorname{Arc}^*(X):=(\tau_{\infty,0}^X)^{-1}(0).
\]

For a closed subset \(\Sigma\subset  X\), endowed with its reduced structure, set
\[
\operatorname{Arc}^{\Sigma}(X)
   =\bigl((\tau_{\infty,0}^X)^{-1}(\Sigma)\bigr)_{\mathrm{red}},
\qquad
J_m^{\Sigma}(X)
   =\bigl((\tau_{m,0}^X)^{-1}(\Sigma)\bigr)_{\mathrm{red}}.
\]
The superscript \(0\) abbreviates \(\Sigma=\{0\}\), so these are arcs
and jets centered at the origin.

An irreducible subset \(C\subset \operatorname{Arc}(X)\) is \emph{non-degenerate} if \( C\not\subset \operatorname{Arc}(\Sing X)\). For the isolated singularities considered here, every nonconstant arc centered at the origin is non-degenerate.

\subsection{Orders of series}
\label{subsec:prelim-orders}

For a series
$a(t)=\sum_{m=0}^{\infty}a_mt^m\in L[[t]]$, define
\[
\ord_t a(t)=\min\{m\geq0:a_m\ne0\},
\qquad \min\emptyset=+\infty.
\]
If $\gamma$ is an arc and $f$ is regular near its center, set
\[
\ord_\gamma(f)=\ord_t(\gamma^*f).
\]
For a scheme point $\gamma\in\operatorname{Arc}(X)$, this means the associated
$\kappa(\gamma)$-valued arc. For a power series, \(\operatorname{ord}\)
without a subscript also denotes \(\ord_t\).

\paragraph{Cylinders and ambient codimension.}
A cylinder in $\operatorname{Arc}(X)$ is the inverse image of a
constructible subset of a finite jet scheme.
Let $S$ be a smooth scheme of finite type over $\mathbb C$,
of pure dimension $d_S$.
For a nonempty cylinder
\[
V=(\tau_{\infty,m}^{S})^{-1}(B),
\]
where $B\subset J_m(S)$ is constructible, define
\[
\operatorname{codim}(V,\operatorname{Arc}(S))
:=(m+1)d_S-\dim B.
\]
Here $\dim B$ denotes the dimension of its Zariski closure in
$J_m(S)$, and $d_S=\dim S$ is the intrinsic Krull dimension of $S$,
not the dimension of an ambient space in which $S$ is embedded.
For $\ell\ge m$, the truncation morphism
$\tau_{\ell,m}^{S}$ is Zariski locally trivial with fiber
$\mathbb A^{(\ell-m)d_S}$, since $S$ is smooth of pure dimension
$d_S$. Consequently,
\[
(\ell+1)d_S
-\dim\bigl((\tau_{\ell,m}^{S})^{-1}(B)\bigr)
=(m+1)d_S-\dim B.
\]
Passing to a common truncation level shows that this definition
is independent of the chosen presentation of $V$ as a cylinder.
The codimension is taken in the whole arc space
$\operatorname{Arc}(S)$; by convention, the empty cylinder has codimension
\(+\infty\); see \cite[Corollary~2.11 and Sections~5 and~9.1]{EM}.

\section{Solution of the Johnson--Kollár Problem A}
\label{sec:solution}

\subsection{The hypersurfaces and their local arc spaces}

Our purpose is to determine
\[\#
\operatorname{Irr}
\bigl(\operatorname{Arc}^{0}(X_{N,\,M})\bigr)
\]
for every \(N\in\{3,4,5\}\) and \(M\geq N\).

We first claim that  \((X_{N,\,M},0)\) is an isolated three-dimensional
cDV hypersurface singularity.

\begin{lemma}
\label{lem:cDV}
For every \(N\in\{3,4,5\}\) and \(M\geq N\), \((X_{N,\,M},0)\) is an isolated three-dimensional cDV hypersurface singularity.

\end{lemma}
\begin{proof}
   
By Reid's local analytic description of three-dimensional cDV singularities \cite[(0.5)]{Reid}, a hypersurface germ of the form
\[
f(x,y,z)+wg(x,y,z,w)=0,
\]
where \(f(x,y,z)=0\) is a Du Val surface singularity, is a cDV singularity.

In our case,
\[
f_{N,\,M}
=
\bigl(x^{2}+y^{3}-z^{N}\bigr)
+
w\bigl(-w^{M-1}\bigr).
\]
For \(N=3,4,5\), the surface
\[
x^{2}+y^{3}-z^{N}=0
\]
is respectively a Du Val singularity of type \(D_{4}\), \(E_{6}\), or \(E_{8}\) \cite{Reid}. Hence \((X_{N,\,M},0)\) is a cDV singularity. In addition, from
\[
\nabla f_{N,\,M}=(2x,\,3y^2,\,-Nz^{N-1},\, -Mw^{M-1})
\]
we get that the singularity is isolated.

\end{proof}

\subsection{Reduction from arc spaces to jet schemes}

We use the following theorem of de Fernex and Wang.

\begin{theorem}[de Fernex--Wang {\cite[Theorem~4]{dFW}}]
\label{thm:dFW4}
Let \(X\) be a variety and \(\Sigma \subset X\) a closed subset. Then, for all \(m\gg0\), the natural map
\[
\Psi_m^\Sigma:
\{
\text{non-degenerate irreducible components of }\operatorname{Arc}^{\Sigma}(X)
\}
\longrightarrow
\operatorname{Irr}\bigl(J_m^{\Sigma}(X)\bigr)
\]
is injective.

More precisely, if \(C\) is a non-degenerate irreducible component of \(\operatorname{Arc}^{\Sigma}(X)\), then \(\Psi_m^\Sigma(C)\) is the unique irreducible component of \(J_m^{\Sigma}(X)\) containing \(p_m(C)\).
\end{theorem}

Therefore, let \(\Sigma=\operatorname{Sing}X\). The natural map\[\Psi_m^\Sigma:
\operatorname{Irr}(\operatorname{Arc}^{\Sigma}(X))
\longrightarrow
\operatorname{Irr}\bigl(J_m^{\Sigma}(X)\bigr)
\]
is defined on every irreducible component because every irreducible component of \(\operatorname{Arc}^{\Sigma}(X)\) is non-degenerate \cite[Remark~6]{dFW}.

\begin{theorem}[de Fernex--Wang {\cite[Theorem~34] {dFW}}]
\label{thm:dFW34}
Let \(X\) be a variety and \(x\in X\) an isolated hcDV singularity. Then, for all \(m\gg0\), the natural map \(\Psi_m^x\) is surjective, and hence bijective.

\end{theorem}

We claim that cDV singularities are hcDV singularities.

\begin{lemma}
\label{lem:cDV-is-hcDV}
The isolated hypersurface cDV singularity \((X,\,0)\) is an hcDV singularity.
    
\end{lemma}
\begin{proof}

Suppose \(X=(f=0)\subset \mathbb{A}^{d+1} \), where \(f\in \mathbb{C}[x_{0},\,\cdots,\,x_{d}]\). The local ring is \[R:=\mathcal{O}_{X,\,0} =\cfrac{\mathbb{C}[x_{0},\,\cdots ,\,x_{d}]_{\mathfrak{m}}}{(f)},\]where \(\mathfrak{m}=(x_{0},\,\cdots,\,x_{d})\subset \mathbb{C}[x_{0},\,\cdots,\,x_{d}]\). Suppose \(\mathfrak{n}=\mathfrak{m}\mathbb{C}[x_{0},\,\cdots,\,x_{d}]_{\mathfrak{m}}/(f)\). Because \(0\) is a singular point, we get \[\cfrac{\partial f}{\partial x_{i}}(0)=0,\qquad \forall i.\]Therefore \(f \in \mathfrak{m}^2\) and \[\cfrac{\mathfrak{n}}{\mathfrak{n}^{2}}\simeq \cfrac{ \mathfrak{m}\mathbb{C}[x_{0},\,\cdots ,\,x_{d}]_{\mathfrak{m}}}{\mathfrak{m}^{2}\mathbb{C}[x_{0},\,\cdots ,\,x_{d}]_{\mathfrak{m}}+(f)}\simeq \cfrac{\mathfrak{m}}{\mathfrak{m}^{2}}.\]We get \[\operatorname{edim}R=\dim _{\mathbb{C}}\mathfrak{m}/\mathfrak{m}^{2}=d+1.\]Therefore \(\operatorname{ecodim}R=\operatorname{edim}R- \dim R = 1\). 
From \cite[Proposition~31]{dFW} the  cDV assumption gives that
\[
\operatorname{mld}_0(X)=d-1
\]
and for every divisor \(E\) over \(X\) computing \(\operatorname{mld}_0    (X)\) we get \(\operatorname{ord}_E(\m_0)=1\) and \(E\) computes \(\operatorname{mld}_0(X,\,(d-2)\{0\})\). Setting \(e=1\) in \cite[Proposition~~33]{dFW}, we get that it is an hcDV singularity.
    
\end{proof}

Therefore, from Lemmas~\ref{lem:cDV} and~\ref{lem:cDV-is-hcDV}, \(X_{N,\,M}\) is an isolated hypersurface hcDV singularity. We can get that

\begin{proposition}[cf. de Fernex--Wang {\cite[Theorem~34]{dFW}}]
\label{prop:jet-component-dimension}
Let \((X,\,0)\) be an isolated three-dimensional hypersurface cDV singularity. Then for \( m\gg 0\) the natural map \[\operatorname{Irr}(\operatorname{Arc}^0(X))\longrightarrow \operatorname{Irr}(J_m^0(X))\]is a bijection and every \(C \in\operatorname{Irr}(J_m^0(X))\) has dimension
\[
\dim C=3m+1.
\]
\end{proposition}
\begin{proof}

It remains to prove the dimension formula. Let \(A=\mathbb{A}^4_{\mathbb{C}}\),
and set \(d=\dim X=3\) and \(e=\operatorname{ecodim}\mathcal{O}_{X,\,0}=1\).
Let \(V\subset\operatorname{Arc}(A)\) be the cylinder over \(C\).
The proof of \cite[Theorem~34]{dFW} gives
\[
\operatorname{codim}(V,\,\operatorname{Arc}(A))=em+d=m+3.
\]
Moreover,
\[
\operatorname{codim}(C,\,J_m(A))=\dim J_m(A)-\dim C=4(m+1)-\dim C.
\]
Therefore,
\[
\dim C=4(m+1)-\operatorname{codim}(C,\,J_m(A))
=4(m+1)-\operatorname{codim}(V,\,\operatorname{Arc}(A))=3m+1.
\]

\end{proof}

\subsection{Exact-order strata}
For an arc \(\gamma=(x,\,y,\,z,\,w)\) we define
\[\operatorname{ord}\gamma:=(\operatorname{ord}x,\,\operatorname{ord}y,\,\operatorname{ord}z,\,\operatorname{ord}w).\]
For \(x=\sum \limits_{i=0}^{m}a_{i}t^{i}\in \mathbb{C}[[t]]/(t^{m+1})\) we define \[\operatorname{ord}_{m}x=\begin{cases}
\min\{i\mid a_{i}\ne 0\},\quad & x\not\equiv 0\pmod{ t^{m+1}},\\ 
m+1,\quad & x \equiv 0\pmod{ t^{m+1}}.
\end{cases}\]and for \(\beta=(x,y,z,w)\in J_m^0(X_{N,\,M})\) we define 
\[\operatorname{ord}_{m}\beta:=(\operatorname{ord}_{m}x,\, \operatorname{ord}_{m}y,\,\operatorname{ord}_{m}z,\,\operatorname{ord}_{m}w).\]

For a weight vector \(v=(a,b,c,d)\in\mathbb Z_{>0}^{4}\) we define
\[
q(v)=\min\{2a,3b,Nc,Md\}
\]
and
\[
\delta(v)=a+b+c+d-q(v).
\]
We denote the $v$-initial form of \(f_{N,\,M}\) by \(\operatorname{in}_v(f_{N,\,M})\), which is the sum of all monomials of
\[
f_{N,\,M}=x^2+y^3-z^N-w^M
\]
whose \(v\)-weight is equal to \(q(v)\).

For \(v\in\{1,\ldots,m+1\}^{4}\) define the exact-order stratum
\[
T_{m,\,v}
=
\left\{
\beta\in J_m^0(X_{N,\,M})
\mid
\operatorname{ord}_m(\beta)=v
\right\}.
\]
We also define
\[
H_v
=
V\bigl(
\operatorname{in}_v(f_{N,\,M})
(X_0,Y_0,Z_0,W_0)
\bigr)
\cap D(X_0Y_0Z_0W_0)
\subset\mathbb A^4.
\]

\subsection{Dimension and irreducible components of the order strata}
\label{subsec:dimension}

\begin{proposition}
\label{prop:general-arc-in-jet-schemes}
Given \(m\gg0\), for every \(C\in\operatorname{Irr}(J_m^0(X_{N,\,M}))\)
there is a nonempty Zariski open subset \(U\subset C\) such that
\[q(\operatorname{ord}_m\beta)\le m,\qquad \forall\beta\in U.\]
    
\end{proposition}
\begin{proof}
Let
\[
v=(a,\,b,\,c,\,d)\in\{1,\,\cdots,\,m+1\}^{4}
\]
and suppose that
\[
q(v)>m.
\]
We first compute \(\dim T_{m,\,v}\).

Suppose that \(v \in \{1,\,\cdots,\, m\}^4\). Since \(q(v)>m\), for every
\[
\beta=(x,\,y,\,z,\,w):=
\left(
\sum\limits_{i=1}^{m}x_{i}t^{i},\,
\sum\limits_{i=1}^{m}y_{i}t^{i},\,
\sum\limits_{i=1}^{m}z_{i}t^{i},\,
\sum\limits_{i=1}^{m}w_{i}t^{i}
\right)
\in J_{m}^{0}(\mathbb{A}^{4})
\]
satisfying \(\operatorname{ord}_{m}\beta=v\), we have
\[
x^{2}\equiv y^{3}\equiv z^{N}\equiv w^{M}
\equiv0\pmod{t^{m+1}}.
\]
Thus
\[
f_{N,\,M}(\beta)\equiv0\pmod{t^{m+1}}
\]
holds automatically. Therefore,
\[
\begin{aligned}
T_{m,\,v}
&=
\left\{
\beta\in J_{m}^{0}(\mathbb{A}^{4})
=\mathbb{A}^{4m}
\mid
\operatorname{ord}_{m}\beta=v
\right\}\\
&=
Z(x_{1},\,\cdots,\,x_{a-1},\,\cdots,\,
w_{1},\,\cdots,\,w_{d-1})
\cap D(x_{a}y_{b}z_{c}w_{d})
\subset\mathbb{A}^{4m},
\end{aligned}
\]
hence
\[
\dim T_{m,\,v}=
4m-(a+b+c+d-4).
\]
When one of the components of \(v\) is \(m+1\), the above argument applies in the same way, and we can get the same expression for \(\dim T_{m,\,v}\). For example, if \(d=m+1\) and \(a,\,b,\,c\le m\), we have \[T_{m,\,v}=Z(x_1,\,\cdots,\, x_{a-1},\,\cdots,\, z_1,\,\cdots,\, z_{c-1})\cap  D(x_ay_bz_c   )\subset \mathbb A^{3m},\]hence \[\dim T_{m,\,v }= 3m-(a+b+c-3)=4m-(a+b+c+d-4).\]

Since \(q(v)>m\), we have
\[
a-1\geq
\left[\cfrac{m}{2}\right],
\qquad
b-1\geq
\left[\cfrac{m}{3}\right],
\qquad
c-1\geq
\left[\cfrac{m}{N}\right].
\]
Therefore,
\[
\dim T_{m,\,v}
\leq
4m-(a+b+c-3)
\leq
4m-
\left(
\left[\cfrac{m}{2}\right]
+
\left[\cfrac{m}{3}\right]
+
\left[\cfrac{m}{N}\right]
\right)
\leq3m.
\]

Define
\[
B_{m}:=
\bigcup\limits_{
v\in\{1,\,\cdots,\,m+1\}^{4},\,
q(v)>m
}
T_{m,\,v}
\subset J_{m}^{0}(X_{N,\,M}).
\]
Since this is a finite union, we have
\[
\dim B_{m}
=
\dim
\overline{B_{m}}^{J_{m}^{0}(X_{N,\,M})}
\leq3m.
\]

Let \(C\in\operatorname{Irr}(J_{m}^{0}(X_{N,\,M}))\).  Since \(\dim C=3m+1\), it is impossible that
\[
C\subset
\overline{B_{m}}^{J_{m}^{0}(X_{N,\,M})}.
\]
Therefore,
\[
U=
C\setminus
\overline{B_{m}}^{J_{m}^{0}(X_{N,\,M})}
\subset C
\]
is a nonempty Zariski open subset. This open subset \(U\) satisfies the required condition.
\end{proof}

\begin{proposition}
\label{prop:dim-of-stratum-q<m}
Given \(m\) and \(v=(a,\,b,\,c,\,d)\in \{1,\,\cdots,\,m+1\}^{4}\) satisfying \(q(v)\le m\), if  \(T_{m,\,v}\ne \emptyset\), then for every \(P\in \operatorname{Irr}(T_{m,\,v})\), we have
\[
\dim P=3m-\delta(v)+3.
\]
\end{proposition}
\begin{proof}
For every \(\beta=(x,\,y,\,z,\,w)\in T_{m,\,v}\), let the exponents be
\[
p_{x}=2,\,  p_{y}=3,\,p_{z}=N,\, p_{w}=M
\]
and let the coefficients be
\[
\varepsilon _x=\varepsilon _y=1,\, \varepsilon_{z}=\varepsilon_{w}=-1.
\]
Write \(v_x=a\), \(v_y=b\), \(v_z=c\), and \(v_w=d\), and set
\[
I(v):=\{u\in\{x,y,z,w\}:v_u\leq m\}.
\]
If \(u\notin I(v)\), then \(v_u=m+1\), so the corresponding
coordinate is identically zero modulo \(t^{m+1}\); its entire block of coefficient variables is omitted. For \(u\in I(v)\), write
\[
u(t)=t^{v_u}\sum_{r=0}^{m-v_u}U_{u,r}t^r,
\qquad U_{u,0}\ne0.
\]
The number of coefficient variables is
\[
K:=\sum_{u\in I(v)}(m-v_u+1)=4m-(a+b+c+d-4),
\]
because each omitted coordinate has \(m-v_u+1=0\).
When all four coordinates belong to \(I(v)\), we use the notation
\(U_{x,r}=X_r\), \(U_{y,r}=Y_r\), \(U_{z,r}=Z_r\), and
\(U_{w,r}=W_r\), so the expansions are
\[
\begin{array}
{ll}x= t^{a}(X_{0}+X_{1}t+\cdots +X_{m-a}t^{m-a}),\qquad & y= t^{b}(Y_{0}+Y_{1}t+\cdots +Y_{m-b}t^{m-b})\\ \\
z= t^{c}(Z_{0}+Z_{1}t+\cdots +Z_{m-c}t^{m-c}),\qquad & w= t^{d}(W_{0}+W_{1}t+\cdots +W_{m-d}t^{m-d})
\end{array}
\]
with \(X_{0}Y_{0}Z_{0}W_{0}\ne 0\).
For arbitrary \(I(v)\), put
\[
g_v:=\prod_{u\in I(v)}U_{u,0},
\qquad
R:=\mathbb C[U_{u,r}\mid u\in I(v),\ 0\leq r\leq m-v_u]_{g_v}.
\]
Thus \(R\) is the coordinate ring of the principal open subset
\(D(g_v)\subset\mathbb A^K\); only the leading coefficients of
the nonzero coordinates are inverted.
Let \(q=q(v)\), and write
\[
f_{N,\,M}(\beta)\equiv t^{q}(F_{0}+F_{1}t+\cdots +F_{m-q}t^{m-q})\pmod{ t^{m+1}}.
\]
Thus
\[
T_{m,\,v}=Z(F_{0},\,\cdots,\,F_{m-q})\cap D(g_v)\subset \mathbb{A}^{K}.
\]

For a weight vector \(\alpha=(\alpha_{x},\,\alpha_{y},\,\alpha_{z},\,\alpha_{w})\in \mathbb{Z}_{>0}^{4}\), we have
\[
\operatorname{in}_{\alpha}(f_{N,\,M})(x,\,y,\,z,\,w)=\sum \limits_{u\in \{x,\,y,\,z,\,w\},\,p_{u}\alpha_{u}=q(\alpha) }\varepsilon_{u}u^{p_{u}}.
\]
Every monomial \(\varepsilon_u u^{p_u}\) occurring in
\(\operatorname{in}_{v}(f_{N,\,M})\) satisfies
\(p_uv_u=q\leq m\), and hence \(u\in I(v)\).
In particular, a coordinate that is identically zero cannot supply
a monomial of minimum weight.
Choose a monomial \(\varepsilon\xi ^{p}\) in
\(\operatorname{in}_{v}(f_{N,\,M})\), and abbreviate
\(U_{\xi,r}\) by \(U_r\). Then
\[
\xi(t)=t^{h}(U_{0}+U_{1}t+\cdots +U_{m-h}t^{m-h}) ,\qquad h=\operatorname{ord}_{m}\xi,\; U_{0} \ne 0.
\]
Since \(q=ph>h\), we have \(m-q<m-h\), so all the variables
\(U_0,\ldots,U_{m-q}\) occur in this coefficient block. Set
\[
D_{\xi}:=\frac{\partial F_0}{\partial U_0}
=\varepsilon pU_0^{p-1}\in R^\times.
\]
Here \(U_0\) is invertible in \(R\), and the ground field has
characteristic zero. Expanding the coefficients gives
\[
F_0=\sum_{\substack{u\in I(v)\\p_uv_u=q}}
\varepsilon_u U_{u,0}^{p_u},
\qquad
F_j=D_{\xi}U_j+G_j\quad(1\leq j\leq m-q),
\]
where \(G_j\) is independent of \(U_j,U_{j+1},\ldots\).
Enumerate the \(K\) coefficient variables belonging to the
coordinates in \(I(v)\) as \(u_1,\ldots,u_K\). Consider
\[
\mathcal{J}_{\mathrm{full}}\vert_{\beta}=\left.\cfrac{\partial (F_{0},\,\cdots ,\,F_{m-q})}{\partial (u_1,\ldots,u_K)} \right\vert_{\beta}.
\]
It contains the submatrix
\[
\mathcal{J}_{\xi}\vert_{\beta}=\left.\cfrac{\partial (F_{0},\,\cdots ,\,F_{m-q})}{\partial (U_{0},\,\cdots ,\,U_{m-q})}\right\vert_{\beta}=\left.\begin{pmatrix}
D_{\xi}\\
* & D_{\xi} \\
* & * & D_{\xi} \\ 
\vdots  & \vdots  & \vdots  & \ddots \\
* & * & * & \cdots  & D_{\xi}
\end{pmatrix} \right\vert_{\beta}.
\]
Thus, \(\det \mathcal{J}_{\xi}\vert_{\beta}\ne 0\), and hence \(\operatorname{rank}(\mathcal{J}_{\mathrm{full}}\vert_{\beta})=m-q+1\). 

Let
\[
A=R/(F_{0},\,\cdots ,\,F_{m-q})
\]
be the coordinate ring of \(T_{m,\,v}\). 

For every closed point \(\beta \in T_{m,\,v}\), let \(\mathfrak{m}_{\beta}\subset R\) and \(\mathfrak{n}_{\beta}=\mathfrak{m}_{\beta}/(F_{0},\,\cdots,\,F_{m-q})\subset A\) be the maximal ideals corresponding to \(\beta\). Consider the local ring
\[
A_{\mathfrak{n}_{\beta}}=\cfrac{R_{\mathfrak{m}_{\beta}}}{(F_{0},\,\cdots ,\,F_{m-q})R_{\mathfrak{m}_{\beta}}}.
\]

We first prove that \(A_{\mathfrak{n}_{\beta}}\) is a regular local ring. The ring \(R_{\mathfrak{m}_{\beta}}\) is a regular local ring satisfying
\[
\dim _{\mathbb{C}} \cfrac{\mathfrak{m}_{\beta}R_{\mathfrak{m}_{\beta}}}{\mathfrak{m}_{\beta}^{2}R_{\mathfrak{m}_{\beta}}}=K.
\]
Moreover,
\[
\cfrac{\mathfrak{n}_{\beta}A_{\mathfrak{n}_{\beta}}}{\mathfrak{n}_{\beta}^{2}A_{\mathfrak{n}_{\beta}}} \simeq\cfrac{\mathfrak{m}_{\beta}R_{\mathfrak{m}_{\beta}}}{\mathfrak{m}_{\beta }^{2}R_{\mathfrak{m}_{\beta}}+(F_{0},\,\cdots ,\,F_{m-q})R_{\mathfrak{m}_{\beta}}}.
\]
We have
\[
F_{i}\equiv \sum \limits_{j=1}^{K}\cfrac{\partial F_{i}}{\partial u_{j}}(\beta)(u_{j}-u_{j}(\beta )) \pmod{ \mathfrak{m}_{\beta}^{2}},\qquad 0 \le i \le m-q.
\]
Since \(\operatorname{rank}(\mathcal{J}_{\mathrm{full}}\vert_{\beta})=m-q+1\), the matrix has full row rank. Let \(\overline{F_{i}}=F_{i}+\mathfrak{m}_{\beta}^{2}\). Then \(\{\overline{F_{0}},\,\cdots,\,\overline{F_{m-q}}\}\) is linearly independent in \(\cfrac{\mathfrak{m}_{\beta}R_{\mathfrak{m}_{\beta}}}{\mathfrak{m}_{\beta}^{2}R_{\mathfrak{m}_{\beta}}}\), and
\[
\operatorname{Span} _{\mathbb{C}} \{\overline{F_{0}},\,\cdots,\,\overline{F_{m-q}}\}=\cfrac{\mathfrak{m}_{\beta }^{2}R_{\mathfrak{m}_{\beta}}+(F_{0},\,\cdots ,\,F_{m-q})R_{\mathfrak{m}_{\beta}}}{\mathfrak{m}_{\beta}^{2}R_{\mathfrak{m}_{\beta}}}.
\]
Therefore,
\[
\cfrac{\mathfrak{n}_{\beta}A_{\mathfrak{n}_{\beta}}}{\mathfrak{n}_{\beta}^{2}A_{\mathfrak{n}_{\beta}}} \simeq\cfrac{\mathfrak{m}_{\beta}R_{\mathfrak{m}_{\beta}}}{\mathfrak{m}_{\beta }^{2}R_{\mathfrak{m}_{\beta}}+(F_{0},\,\cdots ,\,F_{m-q})R_{\mathfrak{m}_{\beta}}}\simeq \cfrac{\mathfrak{m}_{\beta}R_{\mathfrak{m}_{\beta}}/\mathfrak{m}_{\beta}^{2}R_{\mathfrak{m}_{\beta}}}{\operatorname{Span} _{\mathbb{C}} \{\overline{F_{0}},\,\cdots,\,\overline{F_{m-q}}\}}.
\]
Hence,
\[
\operatorname{edim}A_{\mathfrak{n}_{\beta}}= \dim _{\mathbb{C}}\cfrac{\mathfrak{n}_{\beta}A_{\mathfrak{n}_{\beta}}}{\mathfrak{n}_{\beta}^{2}A_{\mathfrak{n}_{\beta}}}  =K-(m-q+1).
\]

For a Noetherian local ring, we have
\[
\dim A_{\mathfrak{n}_{\beta}}\le \operatorname{edim}A_{\mathfrak{n}_{\beta}}=K-(m-q+1).
\]
By Krull's height theorem, if \(\mathfrak{m}_{\beta}\supset \mathfrak{p}\supset (F_{0},\,\cdots,\,F_{m-q})\) is a minimal prime ideal over \((F_0,\, \cdots,\, F_{m-q})\), then
\[
\mathrm{ht}(\mathfrak{p})\le m-q+1.
\]
There is a natural surjection \(A_{\mathfrak{n}_{\beta}}\twoheadrightarrow R_{\mathfrak{m}_{\beta}}/\mathfrak{p}R_{\mathfrak{m}_{\beta}}\). Therefore,
\[
K-(m-q+1)\le \dim R_{\mathfrak{m}_{\beta}}/\mathfrak{p}R_{\mathfrak{m}_{\beta}}\le \dim  A_{\mathfrak{n}_{\beta}}\le K-(m-q+1).
\]
Thus,
\[
\dim A_{\mathfrak{n}_{\beta}}=K-(m-q+1)=\operatorname{edim}A_{\mathfrak{n}_{\beta}}.
\]
Hence \(A_{\mathfrak{n}_{\beta}}\) is a regular local ring. In particular, \(A_{\mathfrak{n}_{\beta}}\) is an integral domain \cite[Tags~0BBZ, 00NN, and~00NP]{Stacks}.

Next, we prove that for every \(P ,\,P'\in \operatorname{Irr}(T_{m,\,v})\) with \(P \ne P'\), we have \(P\cap P'=\emptyset\). Suppose otherwise, and choose a closed point \(p\in P\cap P'\). Such a point exists because \(P\cap P'\) is a nonempty closed subset of a scheme of finite type. Let \(\mathfrak{n}_{p}\subset A\) be the maximal ideal corresponding to \(p\), and let \(\mathfrak{p},\,\mathfrak{p}'\subset A\) be distinct minimal prime ideals such that
\[
P=Z(\mathfrak{p}),\qquad P'=Z(\mathfrak{p}').
\]
Moreover, \(\mathfrak{p},\,\mathfrak{p}'\subset \mathfrak{n}_{p}\). Then \(A_{\mathfrak{n}_{p}}\) has two distinct minimal prime ideals \(\mathfrak{p}A_{\mathfrak{n}_{p}},\,\mathfrak{p}'A_{\mathfrak{n}_{p}}\), contradicting the fact that \(A_{\mathfrak{n}_{p}}\) is an integral domain.

Finally, we prove that every \(P\in \operatorname{Irr}(T_{m,\,v})\) satisfies \(\dim P=3m-\delta(v)+3\). Let \(\mathfrak{q}\subset A\) be the minimal prime ideal corresponding to \(P\). For every closed point \(\beta \in P\), since \(A_{\mathfrak{n}_{\beta}}\) is an integral domain, we have
\[
\mathfrak{q}A_{\mathfrak{n}_{\beta}}=0.
\]
Therefore,
\[
\mathcal{O}_{P,\,\beta}=(A/\mathfrak{q})_{\mathfrak{n}_{\beta}/\mathfrak{q}}\simeq \cfrac{A_{\mathfrak{n}_{\beta}}}{\mathfrak{q}A_{\mathfrak{n}_{\beta}}}=A_{\mathfrak{n}_{\beta}}.
\]
Hence,
\[
\dim P=\dim \mathcal{O}_{P,\,\beta}= K-(m-q+1)=3m-\delta(v)+3.
\]
This completes the proof.
\end{proof}
    
\begin{corollary}
\label{cor:distinct-irr-are-disjoint}
Let \(v=(a,b,c,d)\in\{1,\ldots,m+1\}^4\) satisfy \(q(v)\le m\). If \(P,P'\in\operatorname{Irr}(T_{m,v})\) are distinct,  then
\[P\cap P'=\emptyset.\]
    
\end{corollary}

This follows from the proof of Proposition~\ref{prop:dim-of-stratum-q<m}.

\begin{proposition}
\label{prop:Tmv-and-Hv}
Given \(m\) and \(v=(a,\,b,\,c,\,d)\in \{1,\,\cdots,\,m\}^{4}\) satisfying \(q(v)\le m\), if  \(T_{m,\,v}\ne \emptyset\), then there is a bijection
\[
\operatorname{Irr}(T_{m,\,v})\rightarrow  \operatorname{Irr}(H_{v}).
\]
\end{proposition}
\begin{proof}
We continue to use the notation of Proposition~\ref{prop:dim-of-stratum-q<m}.
Here \(v\in\{1,\ldots,m\}^4\), so all four coordinates belong to
\(I(v)\) and no coefficient block is omitted.

Let
\[
A=\cfrac{R}{(F_{0},\,F_{1},\,\cdots,\,F_{m-q})}
\]
be the coordinate ring of \(T_{m,\,v}\), and let
\[
B=\cfrac{\mathbb{C}[X_{0},\,Y_{0},\,Z_{0},\,W_{0}]_{X_{0}Y_{0}Z_{0}W_{0}}}{(\operatorname{in}_{v}(f_{N,\,M})(X_{0},\,Y_{0},\,Z_{0},\,W_{0}))}=\cfrac{\mathbb{C}[X_{0}^{\pm 1},\,Y_{0}^{\pm 1},\,Z_{0}^{\pm 1},\,W_{0}^{\pm 1}]}{(\operatorname{in}_{v}(f_{N,\,M}))}
\]
be the coordinate ring of \(H_{v}\).

We claim that
\[
A\simeq B[V_{1},\,\cdots ,\,V_{K-4-(m-q)}],
\]
where
\[
\{V_{1},\,\cdots,\,V_{K-4-(m-q)}\}=\{X_{1},\,\cdots,\,X_{m-a},\,\cdots,\,W_{1},\,\cdots,\,W_{m-d}\}\setminus \{U_{1},\,\cdots,\,U_{m-q}\}.
\]
Since we have already obtained
\[
\begin{array}
{l}F_{0}=\operatorname{in}_{v}(f_{N,\,M})(X_{0},\,Y_{0},\,Z_{0},\,W_{0}) \\ \\
F_{j}=D_{\xi}U_{j}+G_{j},\qquad\qquad 1\le j \le m-q,
\end{array}
\]
where \(G_{j}\) is independent of \(U_{j},\,U_{j+1},\,\cdots\), and \(D_{\xi}=\varepsilon pU_{0}^{p-1}\in B^{\times}\), the isomorphism is therefore clear.

Since there are mutually inverse bijections between the minimal prime ideals
\[
\begin{array}
{ll}\operatorname{Min}A\rightarrow  \operatorname{Min}B,\qquad & \mathfrak{p}\mapsto \mathfrak{p}\cap B\\ \\
\operatorname{Min}B\rightarrow  \operatorname{Min}A,\qquad & \mathfrak{P}\mapsto \mathfrak{P}A,
\end{array}
\]
we obtain the bijection
\[
\operatorname{Irr}(T_{m,\,v})\rightarrow  \operatorname{Irr}(H_{v}).
\]
\end{proof}

\begin{corollary}
\label{cor:T-ne-eset}
Given \(m\) and \(v=(a,\,b,\,c,\,d)\in \{1,\,\cdots,\,m\}^{4}\) satisfying \(q(v)\le m\), then
\[
T_{m,\,v}\ne \emptyset \iff H_{v}\ne \emptyset \iff \operatorname{in}_{v}(f_{N,\,M}) \text{ is not a monomial}.
\]
    
\end{corollary}
\begin{proof}

We continue to use the notation of Proposition~\ref{prop:dim-of-stratum-q<m} and Proposition \ref{prop:Tmv-and-Hv}.

\(T_{m,\,v}\ne \emptyset\implies H_v \ne \emptyset\):
In fact, the isomorphism of coordinate rings as \(\mathbb{C}\)-algebras
\[
A\simeq B[V_{1},\,\cdots ,\,V_{K-4-(m-q)}],
\]
and the equivalence between the category of affine schemes and the opposite category of commutative rings  \cite[Tag~01I2]{Stacks} give the isomorphism
\[
T_{m,\,v}\simeq H_{v} \times  \mathbb{A}^{K-4-(m-q)}.
\]
This proves the implication \(T_{m,\,v}\ne\emptyset\implies H_v\ne\emptyset\).

\(H_v \ne \emptyset \implies T_{m,\,v}\ne \emptyset\):
Let \((X_0,\,Y_0,\,Z_0,\,W_0)\in H_v\). We construct an \(m\)-jet \(\beta=(x,\,y,\,z,\,w)\in T_{m,\,v}\). Write
\[
\begin{array}
{ll}x= t^{a}(X_{0}+X_{1}t+\cdots +X_{m-a}t^{m-a}),\qquad & y= t^{b}(Y_{0}+Y_{1}t+\cdots +Y_{m-b}t^{m-b})\\ \\
z= t^{c}(Z_{0}+Z_{1}t+\cdots +Z_{m-c}t^{m-c}),\qquad & w= t^{d}(W_{0}+W_{1}t+\cdots +W_{m-d}t^{m-d}).
\end{array}
\]
Let \(\varepsilon\xi^p\) be a monomial in \(\operatorname{in}_v(f_{N,\,M})(x,\,y,\,z,\,w)\), \(\xi \in \{x,\,y,\,z,\,w\}\). Write \[\xi = t^h(U_0+ U_1t+ \cdots + U_{m-h}t^{m-h}).\]
Then \(q:=q(v)=ph\). Let
\[
f_{N,M}(x,y,z,w)
\equiv t^q(F_0+F_1t+\cdots+F_{m-q}t^{m-q})
\pmod{t^{m+1}}.
\]
By our choice of the leading coefficients, we have
\[
F_0=\operatorname{in}_v(f_{N,\,M})(X_0,\,Y_0,\,Z_0,\,W_0)=0.
\]
Expanding \(\varepsilon\xi^p\), we see that its coefficient of \(t^{j+q}\)
contains exactly one term involving \(U_j\), namely
\[\varepsilon pU_0^{p-1}U_j.\]
All other terms in this coefficient involve only \(U_0,\ldots,U_{j-1}\).
Consequently,
\[
F_j=D_\xi U_j+G_j,
\qquad
D_\xi=\varepsilon pU_0^{p-1}\ne0,
\]where \(G_j\) is independent of \(U_j,U_{j+1},\cdots\).

Now set all higher-order coefficients of the other three coordinates equal to zero, and define recursively
\[
U_j=-D_\xi^{-1}G_j,\qquad 1\le j\le m-q.
\]At the \(j\)-th step, all of \(U_1,\ldots,U_{j-1}\) appearing in \(G_j\) have already been determined. Since \(D_\xi\ne0\), the recursion is well defined and ensures that \(F_j=0\). Since \(q=ph\ge h\) we get \[m-q\le m-h.\]It then suffices to set all remaining higher-order coefficients equal to zero.

The quadruple \(\beta=(x,\,y,\,z,\,w)\) constructed in this way satisfies
\[
F_0=F_1=\cdots=F_{m-q}=0,
\]
and hence
\[
f_{N,M}(x,y,z,w)\equiv0\pmod{t^{m+1}}.
\]
Moreover, the four leading coefficients remain nonzero, so
\[
\operatorname{ord}_m(x,y,z,w)=(a,b,c,d).
\]
Therefore, \((x,y,z,w)\in T_{m,v}\), proving that \(H_v \ne \emptyset \implies T_{m,\,v}\ne \emptyset\).

\(H_v\ne \emptyset \iff\operatorname{in}_{v}(f_{N,\,M}) \text{ is not a monomial}\):
The coordinate ring of \(H_v\) is 
\[
B=\cfrac{\mathbb{C}[X_{0},\,Y_{0},\,Z_{0},\,W_{0}]_{X_{0}Y_{0}Z_{0}W_{0}}}{(\operatorname{in}_{v}(f_{N,\,M})(X_{0},\,Y_{0},\,Z_{0},\,W_{0}))}=\cfrac{\mathbb{C}[X_{0}^{\pm 1},\,Y_{0}^{\pm 1},\,Z_{0}^{\pm 1},\,W_{0}^{\pm 1}]}{(\operatorname{in}_{v}(f_{N,\,M}))},
\]
Thus, \(H_v\ne\emptyset\iff B\ne0\iff\operatorname{in}_v(f_{N,\,M})\notin\mathbb{C}[X_0^{\pm1},\,Y_0^{\pm1},\,Z_0^{\pm1},\,W_0^{\pm1}]^\times\),
and the latter condition means that \(\operatorname{in}_v(f_{N,M})\)
is not a Laurent monomial.

\end{proof}

\subsection{The component criterion}
\label{subsec:component-criterion}

In fact, according to the above argument, we only need to consider the weight vector \(v\) satisfying \(v \in \{1,\, \cdots,\, m\}^4\) when \(q(v) \le m\) and \(\delta(v    )=2\).

\begin{lemma}
\label{lem:comp < m + 1}
Given \(m\) and \(v=(a,\,b,\,c,\,d)\in \{1,\,\cdots,\,m+1\}^{4}\) satisfying \(q(v)\le m\) and \(\delta(v)=2\), then \(v \in \{1,\,\cdots,\, m\}^4\) holds automatically.
    
\end{lemma}
\begin{proof}

Without loss of generality, assume that \(d=m+1\), then \[\delta(v ) =a+b+c+d-q(v) \ge m+4-q(v)\ge 4.\]
This is a contradiction.
    
\end{proof}

In summary, we have the theorem for finding an irreducible component as follows.

\begin{theorem}
\label{thm:determine-irr}
For sufficiently large \(m\), let
\[
\mathcal{P}_{m}=\{(v,\,P)\mid v\in \{1,\,\cdots ,\,m+1\}^{4},\,q(v)\le m,\,\delta(v)=2,\,P\in \operatorname{Irr}(T_{m,\,v})\}.
\]
Then there is a natural bijection
\[
\Phi _{m}:\mathcal{P}_{m}\rightarrow  \operatorname{Irr}(J_{m}^{0}(X_{N,\,M})),\qquad (v,\,P)\mapsto \overline{P}^{J_{m}^{0}(X_{N,\,M})}.
\]
\end{theorem}
\begin{proof}
From Lemma \ref{lem:comp < m + 1} \(v\in \{1,\,\cdots,\, m\}^4\) automatically holds. We first show that \(\Phi _{m}\) is well-defined.

For \((v,\,P)\in \mathcal{P}_{m}\), we have
\[
\dim P=\dim \overline{P}^{J_{m}^{0}(X_{N,\,M})} =3m-\delta(v)+3=3m+1.
\]
There exists \(C\in \operatorname{Irr}(J_{m}^{0}(X_{N,\,M}))\) such that
\[
\overline{P}^{J_{m}^{0}(X_{N,\,M})} \subset C.
\]
Since the two sides have the same dimension \(3m+1\) and \(\overline{P}^{J_{m}^{0}(X_{N,\,M})}\) and \( C\) are both irreducible, equality holds.

We next show that \(\Phi _{m}\) is surjective.

Let \(C\in \operatorname{Irr}(J_{m}^{0}(X_{N,\,M}))\). There exists a nonempty Zariski open subset \(U\subset C\) such that
\[
q(\operatorname{ord}_{m}\beta )\le m,\qquad  \forall \beta \in  U.
\]
Write
\[
\bigcup \limits_{v\in \{1,\,\cdots ,\,m+1\}^{4},\,q(v)\le m} \operatorname{Irr}(T_{m,\,v})=\{P_{1},\,\cdots ,\,P_{L}\}.
\]
Then
\[
U\subset \bigcup \limits_{v\in \{1,\,\cdots ,\,m+1\}^{4},\,q(v)\le m} T_{m,\,v}=\bigcup \limits_{i=1}^{L}P_{i}.
\]
Thus \(U=\bigcup \limits_{i=1}^{L}(P_{i}\cap U)\), and hence
\[
C=\overline{U}^{C}=\bigcup \limits_{i=1}^{L}\overline{P_{i}\cap U}^{C}.
\]
Since \(C\) is irreducible, there exists \(i_{0}\) such that
\[
C=\overline{P_{i_{0}}\cap U}^{C} \subset \overline{P_{i_{0}}}^{J_{m}^{0}(X_{N,\,M})}.
\]
Since \(P_{i_{0}}\) is irreducible, \(\overline{P_{i_{0}}}^{J_{m}^{0}(X_{N,\,M})}\) is irreducible. Since \(C\) is an irreducible component, we obtain
\[
C=\overline{P_{i_{0}}}^{J_{m}^{0}(X_{N,\,M})}.
\]
Write \(P_{i_{0}}\in \operatorname{Irr}(T_{m,\,v_{0}})\), where \(q(v_{0})\le m\). Since
\[
\dim P_{i_{0}}=3m-\delta(v_{0})+3=\dim C=3m+1,
\]
we obtain \(\delta(v_0)=2\). Therefore, \(\Phi _{m}\) is surjective.

We finally show that \(\Phi _{m}\) is injective. For a fixed \(v\), write
\[
\operatorname{Irr}(T_{m,\,v}) =\{P,\,P_{1},\,\cdots ,\,P_{r}\}.
\]
By Corollary~\ref{cor:distinct-irr-are-disjoint}, \(P\cap P_i=\emptyset\) for every \(i\). Thus \(P=T_{m,\,v}\setminus \bigcup \limits_{i=1}^{r}P_{i}\) is both open and closed in \(T_{m,\,v}\). The condition \(\operatorname{ord}_{m}\beta=v\) implies that \(T_{m,\,v}\) is an open subset of a closed set; namely,
\[
T_{m,\,v}=C\cap V,
\]
where \(C\) and \(V\) are, respectively, a closed subset and an open subset of \(J_{m}^{0}(X_{N,\,M})\). Hence \(P\) is open in \(C\). Since \(\overline{P}^{J_{m}^{0}(X_{N,\,M})}\subset C\), it follows that \(P\) is open in \(\overline{P}^{J_{m}^{0}(X_{N,\,M})}\).

Let \((v,\,P),\,(v',\,P')\in \mathcal{P}_{m}\) satisfy
\[
\Phi _{m}(v,\,P)=\Phi _{m}(v',\,P'):=D.
\]
Then \(P\) and \(P'\) are open subsets of \(D\). Since \(D\) is irreducible, we have
\[
P\cap P'\ne \emptyset.
\]
If \(v\ne v'\), then \(T_{m,\,v}\cap T_{m,\,v'}= \emptyset\), which is impossible. If \(v=v'\), Corollary~\ref{cor:distinct-irr-are-disjoint} shows that the distinct irreducible components of \(T_{m,\,v}\) are disjoint. Therefore, we must have
\[
v=v',\qquad P=P'.
\]
This completes the proof.
\end{proof}

\begin{corollary}
    \label{cor:determine-irr-N}
Suppose
\[
\mathcal{V}_{N,\,M}=\{v\in  \mathbb{Z}_{>0}^{4}\mid \delta(v)=2,\, \operatorname{in}_{v}(f_{N,\,M})\text{ is not a monomial}\}
\]
then
\[
\#\operatorname{Irr}\bigl(\operatorname{Arc}^{0}(X_{N,\,M})\bigr)=\sum \limits_{v\in \mathcal{V}_{N,\,M}}\#\bigl\{\text{irreducible factors of }\operatorname{in}_{v}(f_{N,\,M}) \text{ in }\mathbb{C}[X_{0},\,Y_{0},\,Z_{0},\,W_{0}]\bigr\}
.\]
\end{corollary}
\begin{proof}

From Theorem~\ref{thm:determine-irr}, Proposition~\ref{prop:Tmv-and-Hv} and Corollary \ref{cor:T-ne-eset}, for \(m\gg 0\) we have
\begin{align*}
    \#\operatorname{Irr}(\operatorname{Arc}^0(X_{N,\,M}))={}&\#\operatorname{Irr}(J_m^0(X_{N,\,M}))=\#\mathcal{P}_m\\
    ={}&\sum\limits_{\substack{
v\in\{1,\ldots,m\}^{4}\\
q(v)\le m,\ \delta(v)=2
}}\#\operatorname{Irr}(T_{m,\,v})\\
    ={}&\sum\limits_{\substack{v\in \{1,\,\cdots ,\,m\}^{4},\,q(v)\le m,\,\delta(v)=2\\\operatorname{in}_v(f_{N,\,M})\text{ is not monomial}}}\#\operatorname{Irr}(H_v).
\end{align*}
Thus, letting \(m\to \infty\) we have
\[
\#\operatorname{Irr}\bigl(\operatorname{Arc}^{0}(X_{N,\,M})\bigr)=\sum \limits_{v\in \mathcal{V}_{N,\,M}}\#
\operatorname{Irr}(H_v).
\]Since the non-monomial initial form \(\operatorname{in}_v(f_{N,M})\) is not divisible by any coordinate variable, each of its irreducible factors defines an irreducible
hypersurface meeting \((\mathbb C^\times)^4\). The intersections of these hypersurfaces with the torus are precisely the irreducible components of \(H_v\), we get
\[
\#\operatorname{Irr}\bigl(\operatorname{Arc}^{0}(X_{N,\,M})\bigr)=\sum \limits_{v\in \mathcal{V}_{N,\,M}}\#\bigl\{\text{irreducible factors of }\operatorname{in}_{v}(f_{N,\,M}) \text{ in }\mathbb{C}[X_{0},\,Y_{0},\,Z_{0},\,W_{0}]\bigr\}.
\]
    
\end{proof}

\subsection{Enumerations and factorizations}
\label{subsec:enum-fact}
We first show that the enumeration required by Corollary~\ref{cor:determine-irr-N} is finite.

\begin{lemma}
\label{lem:q-upper-bound}
Let
\[
v=(a,\,b,\,c,\,d)\in\mathcal{V}_{N,\,M}
\]
and write \(q=q(v)\). Then
\[
q\leq
\begin{cases}
6,&N=3,\\
12,&N=4,\\
30,&N=5.
\end{cases}
\]
\end{lemma}
\begin{proof}
By the definition of \(q(v)\), we have
\[
a\geq\left\lceil\frac{q}{2}\right\rceil,
\qquad
b\geq\left\lceil\frac{q}{3}\right\rceil,
\qquad
c\geq\left\lceil\frac{q}{N}\right\rceil,
\qquad
d\geq1.
\]
Since \(\delta(v)=2\), we have
\[
a+b+c+d=q+2.
\]
Consequently,
\[
q+2
\geq
\left\lceil\frac{q}{2}\right\rceil
+
\left\lceil\frac{q}{3}\right\rceil
+
\left\lceil\frac{q}{N}\right\rceil
+1
\geq
\frac{q}{2}+\frac{q}{3}+\frac{q}{N}+1.
\]
It follows that
\[
q\left(
\frac{1}{2}+\frac{1}{3}+\frac{1}{N}-1
\right)
\leq1.
\]
For \(N=3,4,5\), the coefficient on the left-hand side is respectively
\[
\frac{1}{6},
\qquad
\frac{1}{12},
\qquad
\frac{1}{30}.
\]
The stated bounds follow.
\end{proof}

The finite enumeration described above was carried out using \textsc{Singular}, version \(4.0.2\) \cite{Singular}. The program checks every integer \(q\) in the range provided by Lemma~\ref{lem:q-upper-bound} and every possible quadruple \((a,\,b,\,c,\,d)\). Notice also that if \(M>q_{\max}\),
then
\[
Md>q
\]
for every \(d\geq1\) and every \(q\leq q_{\max}\). Therefore, the case \(M=q_{\max}+1\) represents every \(M\geq q_{\max}+1\).

The resulting weight vectors are listed in Table~\ref{tab:relevant-weight-vectors}.

\begingroup
\small
\setlength{\tabcolsep}{4pt}
\setlength{\LTleft}{\fill}
\setlength{\LTright}{\fill}
\renewcommand{\arraystretch}{1.15}

\begin{longtable}{
>{$}c<{$}|
>{$}c<{$}|
>{$}c<{$}|
>{\centering\arraybackslash$}p{0.48\textwidth}<{$}
}
\caption{The relevant weight vectors for all values of \(N\) and \(M\).}
\label{tab:relevant-weight-vectors}\\

\hline
N & q_{\max} & \text{Range of }M
& \text{Candidates }(v;q(v))\\
\hline
\endfirsthead

\multicolumn{4}{c}{
\tablename\ \thetable\space (continued)
}\\
\hline
N & q_{\max} & \text{Range of }M
& \text{Candidates }(v;q(v))\\
\hline
\endhead

\hline
\multicolumn{4}{r}{Continued on the next page}
\\
\endfoot

\hline
\endlastfoot

3&6&3\leq M\leq5&
\bigl((2,1,1,1);3\bigr)
\\ \hline

3&6&M\geq6&
\begin{gathered}
\bigl((2,1,1,1);3\bigr),\\
\bigl((3,2,2,1);6\bigr)
\end{gathered}
\\ \hline

4&12&4\leq M\leq5&
\bigl((2,2,1,1);4\bigr)
\\ \hline

4&12&6\leq M\leq7&
\begin{gathered}
\bigl((2,2,1,1);4\bigr),\\
\bigl((3,2,2,1);6\bigr)
\end{gathered}
\\ \hline

4&12&8\leq M\leq11&
\begin{gathered}
\bigl((2,2,1,1);4\bigr),\\
\bigl((3,2,2,1);6\bigr),\\
\bigl((4,3,2,1);8\bigr)
\end{gathered}
\\ \hline

4&12&M\geq12&
\begin{gathered}
\bigl((2,2,1,1);4\bigr),\\
\bigl((3,2,2,1);6\bigr),\\
\bigl((4,3,2,1);8\bigr),\\
\bigl((6,4,3,1);12\bigr)
\end{gathered}
\\ \hline

5&30&M=5&
\bigl((3,2,1,1);5\bigr)
\\ \hline

5&30&6\leq M\leq7&
\bigl((3,2,2,1);6\bigr)
\\ \hline

5&30&M=8&
\begin{gathered}
\bigl((3,2,2,1);6\bigr),\\
\bigl((4,3,2,1);8\bigr)
\end{gathered}
\\ \hline

5&30&M=9&
\begin{gathered}
\bigl((3,2,2,1);6\bigr),\\
\bigl((5,3,2,1);9\bigr)
\end{gathered}
\\ \hline

5&30&10\leq M\leq11&
\begin{gathered}
\bigl((3,2,2,1);6\bigr),\\
\bigl((5,4,2,1);10\bigr)
\end{gathered}
\\ \hline

5&30&12\leq M\leq13&
\begin{gathered}
\bigl((3,2,2,1);6\bigr),\\
\bigl((5,4,2,1);10\bigr),\\
\bigl((6,4,3,1);12\bigr)
\end{gathered}
\\ \hline

5&30&M=14&
\begin{gathered}
\bigl((3,2,2,1);6\bigr),\\
\bigl((5,4,2,1);10\bigr),\\
\bigl((6,4,3,1);12\bigr),\\
\bigl((7,5,3,1);14\bigr)
\end{gathered}
\\ \hline

5&30&15\leq M\leq17&
\begin{gathered}
\bigl((3,2,2,1);6\bigr),\\
\bigl((5,4,2,1);10\bigr),\\
\bigl((6,4,3,1);12\bigr),\\
\bigl((8,5,3,1);15\bigr)
\end{gathered}
\\ \hline

5&30&18\leq M\leq19&
\begin{gathered}
\bigl((3,2,2,1);6\bigr),\\
\bigl((5,4,2,1);10\bigr),\\
\bigl((6,4,3,1);12\bigr),\\
\bigl((8,5,3,1);15\bigr),\\
\bigl((9,6,4,1);18\bigr)
\end{gathered}
\\ \hline

5&30&20\leq M\leq23&
\begin{gathered}
\bigl((3,2,2,1);6\bigr),\\
\bigl((5,4,2,1);10\bigr),\\
\bigl((6,4,3,1);12\bigr),\\
\bigl((8,5,3,1);15\bigr),\\
\bigl((9,6,4,1);18\bigr),\\
\bigl((10,7,4,1);20\bigr)
\end{gathered}
\\ \hline

5&30&24\leq M\leq29&
\begin{gathered}
\bigl((3,2,2,1);6\bigr),\\
\bigl((5,4,2,1);10\bigr),\\
\bigl((6,4,3,1);12\bigr),\\
\bigl((8,5,3,1);15\bigr),\\
\bigl((9,6,4,1);18\bigr),\\
\bigl((10,7,4,1);20\bigr),\\
\bigl((12,8,5,1);24\bigr)
\end{gathered}
\\ \hline

5&30&M\geq30&
\begin{gathered}
\bigl((3,2,2,1);6\bigr),\\
\bigl((5,4,2,1);10\bigr),\\
\bigl((6,4,3,1);12\bigr),\\
\bigl((8,5,3,1);15\bigr),\\
\bigl((9,6,4,1);18\bigr),\\
\bigl((10,7,4,1);20\bigr),\\
\bigl((12,8,5,1);24\bigr),\\
\bigl((15,10,6,1);30\bigr)
\end{gathered}
\\

\end{longtable}
\endgroup

Absolute factorizations of \(\operatorname{in}_v(f_{N,\,M})\) were calculated using \textsc{Singular}, version \(4.0.2\), and the absFactorize procedure \texttt{absFactorize} from \texttt{absfact.lib}.

\begingroup
\footnotesize
\setlength{\tabcolsep}{3pt}
\setlength{\LTleft}{\fill}
\setlength{\LTright}{\fill}
\renewcommand{\arraystretch}{1.15}

\begin{longtable}{
>{$}c<{$}|
>{\centering\arraybackslash$}p{0.13\textwidth}<{$}|
>{\centering\arraybackslash$}p{0.17\textwidth}<{$}|
>{$}c<{$}|
>{\centering\arraybackslash$}p{0.31\textwidth}<{$}|
>{$}c<{$}
}
\caption{Initial forms and numbers of absolute irreducible factors.}
\label{tab:absolute-factorizations}\\

\hline
N
&
\text{Range of }M
&
v
&
q(v)
&
\operatorname{in}_{v}(f_{N,\,M})
&
\#\text{ factors}
\\
\hline
\endfirsthead

\multicolumn{6}{c}{
\tablename\ \thetable\space (continued)
}\\
\hline
N
&
\text{Range of }M
&
v
&
q(v)
&
\operatorname{in}_{v}(f_{N,\,M})
&
\#\text{ factors}
\\
\hline
\endhead

\hline
\multicolumn{6}{r}{Continued on the next page}
\\
\endfoot

\hline
\endlastfoot

3
&
M=3
&
(2,1,1,1)
&
3
&
y^{3}-z^{3}-w^{3}
&
1
\\ \hline

3
&
M\geq4
&
(2,1,1,1)
&
3
&
y^{3}-z^{3}
&
3
\\ \hline

3
&
M=6
&
(3,2,2,1)
&
6
&
x^{2}+y^{3}-z^{3}-w^{6}
&
1
\\ \hline

3
&
M\geq7
&
(3,2,2,1)
&
6
&
x^{2}+y^{3}-z^{3}
&
1
\\ \hline

4
&
M=4
&
(2,2,1,1)
&
4
&
x^{2}-z^{4}-w^{4}
&
1
\\ \hline

4
&
M\geq5
&
(2,2,1,1)
&
4
&
x^{2}-z^{4}
&
2
\\ \hline

4
&
M=6
&
(3,2,2,1)
&
6
&
x^{2}+y^{3}-w^{6}
&
1
\\ \hline

4
&
M\geq7
&
(3,2,2,1)
&
6
&
x^{2}+y^{3}
&
1
\\ \hline

4
&
M=8
&
(4,3,2,1)
&
8
&
x^{2}-z^{4}-w^{8}
&
1
\\ \hline

4
&
M\geq9
&
(4,3,2,1)
&
8
&
x^{2}-z^{4}
&
2
\\ \hline

4
&
M=12
&
(6,4,3,1)
&
12
&
x^{2}+y^{3}-z^{4}-w^{12}
&
1
\\ \hline

4
&
M\geq13
&
(6,4,3,1)
&
12
&
x^{2}+y^{3}-z^{4}
&
1
\\ \hline

5
&
M=5
&
(3,2,1,1)
&
5
&
-z^{5}-w^{5}
&
5
\\ \hline

5
&
M=6
&
(3,2,2,1)
&
6
&
x^{2}+y^{3}-w^{6}
&
1
\\ \hline

5
&
M\geq7
&
(3,2,2,1)
&
6
&
x^{2}+y^{3}
&
1
\\ \hline

5
&
M=8
&
(4,3,2,1)
&
8
&
x^{2}-w^{8}
&
2
\\ \hline

5
&
M=9
&
(5,3,2,1)
&
9
&
y^{3}-w^{9}
&
3
\\ \hline

5
&
M=10
&
(5,4,2,1)
&
10
&
x^{2}-z^{5}-w^{10}
&
1
\\ \hline

5
&
M\geq11
&
(5,4,2,1)
&
10
&
x^{2}-z^{5}
&
1
\\ \hline

5
&
M=12
&
(6,4,3,1)
&
12
&
x^{2}+y^{3}-w^{12}
&
1
\\ \hline

5
&
M\geq13
&
(6,4,3,1)
&
12
&
x^{2}+y^{3}
&
1
\\ \hline

5
&
M=14
&
(7,5,3,1)
&
14
&
x^{2}-w^{14}
&
2
\\ \hline

5
&
M=15
&
(8,5,3,1)
&
15
&
y^{3}-z^{5}-w^{15}
&
1
\\ \hline

5
&
M\geq16
&
(8,5,3,1)
&
15
&
y^{3}-z^{5}
&
1
\\ \hline

5
&
M=18
&
(9,6,4,1)
&
18
&
x^{2}+y^{3}-w^{18}
&
1
\\ \hline

5
&
M\geq19
&
(9,6,4,1)
&
18
&
x^{2}+y^{3}
&
1
\\ \hline

5
&
M=20
&
(10,7,4,1)
&
20
&
x^{2}-z^{5}-w^{20}
&
1
\\ \hline

5
&
M\geq21
&
(10,7,4,1)
&
20
&
x^{2}-z^{5}
&
1
\\ \hline

5
&
M=24
&
(12,8,5,1)
&
24
&
x^{2}+y^{3}-w^{24}
&
1
\\ \hline

5
&
M\geq25
&
(12,8,5,1)
&
24
&
x^{2}+y^{3}
&
1
\\ \hline

5
&
M=30
&
(15,10,6,1)
&
30
&
x^{2}+y^{3}-z^{5}-w^{30}
&
1
\\ \hline

5
&
M\geq31
&
(15,10,6,1)
&
30
&
x^{2}+y^{3}-z^{5}
&
1
\\\end{longtable}
\endgroup
\Needspace{6\baselineskip}
\subsection{Proof of Theorem \ref{mta}}
\label{subsec:final-proof}

\begin{proof}
For each pair \((N,M)\), substitute the lists obtained in
Subsection~\ref{subsec:enum-fact} into the formula of
Corollary~\ref{cor:determine-irr-N}. This gives precisely the values
listed in Theorem~\ref{mta}.
\end{proof}

\section{The Eight Open Families in the Arc Space of the
\texorpdfstring{$E_8$}{E8} Surface}
\label{sec:e8-open-families}

We turn to Theorem~\ref{mtb}, which concerns the surface \(X_5\)
rather than the three-dimensional hypersurfaces of
Section~\ref{sec:solution}. We first determine the resolution data
and identify the exact-order families by transversality. We then
prove their separation using corrected curve selection and formal
wedge adjacency, establish Zariski openness in the centered arc
scheme, and deduce openness of the complex point sets in the
coefficientwise product topology. Irreducible components and generic
points are understood in the scheme-theoretic Zariski topology.
Closures are also taken in the Zariski topology unless explicitly
marked by the superscript $\mathrm{prod}$, which denotes closure
in the coefficientwise product topology on $\mathcal A(\mathbb C)$.

\subsection{The surface and the exact-order families}
\label{subsec:e8-setup}

Set
\[
X:=X_5=\{x^2+y^3=z^5\}\subset\mathbb A^3
\]
and define the reduced centered arc scheme
\[
\mathcal A:=\operatorname{Arc}^{0}(X)
=\bigl((\tau_{\infty,0}^X)^{-1}(0)\bigr)_{\mathrm{red}}.
\]
Thus \(\mathcal A\) is a closed reduced subscheme of
\(\operatorname{Arc}(X)\). Every field-valued point of \(\mathcal A\)
has coordinate expansions
\[
x(t)=\sum_{m=1}^{\infty}x_mt^m,
\qquad
y(t)=\sum_{m=1}^{\infty}y_mt^m,
\qquad
z(t)=\sum_{m=1}^{\infty}z_mt^m.
\]
Here \(x_m,y_m,z_m\) also denote the corresponding coefficient
functions on \(\mathcal A\). The sums begin with \(m=1\) because the
arc is centered at the origin. Thus every nonzero coordinate series
has order at least one, whereas an identically zero coordinate series
has order \(+\infty\). The constant zero arc is denoted by
\(0_\infty\).

For vectors
\(u,w\in(\mathbb Z_{\geq0}\cup\{+\infty\})^3\), write
\[
u\leq w
\quad\Longleftrightarrow\quad
u_k\leq w_k\quad(k=1,2,3),
\]
where every finite integer is smaller than \(+\infty\); write
\(u<w\) when \(u\leq w\) and \(u\ne w\). All vector inequalities in
this section refer to this coordinatewise order.

The eight order vectors, numbered in the order of the corresponding
families in Johnson--Koll\'ar's Example~14
\cite[pp.~530--531]{JK}, are
\[
\begin{array}{c|cccccccc}
i&1&2&3&4&5&6&7&8\\ \midrule
v_i&(15,10,6)&(3,2,2)&(6,4,3)&(9,6,4)&
(12,8,5)&(5,4,2)&(10,7,4)&(8,5,3).
\end{array}
\]
For \(v_i=(p_i,q_i,r_i)\), set
\[
Z_i:=\{\alpha\in\mathcal A:
\ord_\alpha(x)\geq p_i,\ \ord_\alpha(y)\geq q_i,
\ \ord_\alpha(z)\geq r_i\},
\qquad
h_i:=x_{p_i}y_{q_i}z_{r_i},
\]
and
\begin{equation}\label{eq:e8-Wi-definition}
W_i:=Z_i\cap D_{\mathcal A}(h_i)
=\{\alpha\in\mathcal A:
(\ord_\alpha x,\ord_\alpha y,\ord_\alpha z)=v_i\}.
\end{equation}
Here
\(D_{\mathcal A}(h_i)=\{\alpha\in\mathcal A:h_i(\alpha)\ne0\}\).
We give \(Z_i\) its reduced closed subscheme structure and \(W_i\)
the induced reduced locally closed structure. In particular,
\(W_i(\mathbb C)\) is the set of complex arcs with exact order vector
\(v_i\); no closure is taken in this definition.

More explicitly, if \(M_i=\max\{p_i,q_i,r_i\}\), then
\[
Z_i=(\tau_{\infty,M_i}^X)^{-1}(B_i),
\]
where \(B_i\subset J_{M_i}^0(X)\) is cut out by the finitely many
coefficient vanishings below \(p_i,q_i,r_i\). Hence \(Z_i\) is a
closed cylinder and \(W_i\) is Zariski open in \(Z_i\).
Subsection~\ref{subsec:e8-product-openness} proves the stronger
assertion that \(W_i\) is Zariski open in \(\mathcal A\), and
deduces the coefficientwise product-topology openness of
\(W_i(\mathbb C)\) in \(\mathcal A(\mathbb C)\).

For every field extension \(F/\mathbb C\), put
\[
\mathcal A(F):=\operatorname{Hom}_{\mathbb C}
 (\operatorname{Spec}F,\mathcal A),
\qquad
W_i(F):=\operatorname{Hom}_{\mathbb C}
 (\operatorname{Spec}F,W_i).
\]
Equivalently, \(W_i(F)\) consists of the arcs
\(\operatorname{Spec}F[[t]]\to X_F\), centered at the origin, whose
three coordinate orders form \(v_i\).

We record the geometric property required after scalar extension.

\begin{lemma}[Geometric normality]
\label{lem:e8-geometric-normality}
For every field extension \(F/\mathbb C\), the surface
\[
X_F:=X\times_{\mathbb C}\operatorname{Spec}F
\]
is integral and normal, and the reduced support of its singular locus
is the origin.
\end{lemma}

\begin{proof}
In the rational function field \(F(z)(y)\), the element \(z^5-y^3\)
has valuation \(-3\) at the place \(y=\infty\), and hence is not a
square. Therefore \(x^2-(z^5-y^3)\) is irreducible in \(F(y,z)[x]\),
and Gauss's lemma gives irreducibility in \(F[y,z][x]\). Thus \(X_F\)
is integral. Its Jacobian ideal is generated by
\(2x,3y^2,-5z^4\), so the reduced singular locus is the origin. The
hypersurface \(X_F\) is Cohen--Macaulay and satisfies \((S_2)\). Its
singular locus has codimension two, so it is regular in codimension
one. Serre's criterion proves normality.
\end{proof}

Let
\[
\pi:Y\longrightarrow X
\]
be the minimal resolution, and write
\[
\pi_{\mathrm{arc}}:=\operatorname{Arc}(\pi):
\operatorname{Arc}(Y)\longrightarrow\operatorname{Arc}(X).
\]
Thus \(Y\) is smooth, \(\pi\) is proper and birational, and \(\pi\)
is an isomorphism over \(X\setminus\{0\}\). For a prime exceptional
curve \(E\subset Y\), the local ring at its generic point is a
discrete valuation ring. If \(u\) is a uniformizer and
\(0\ne f\in\mathbb C(X)\), write
\[
\pi^*f=u^\nu\varepsilon,
\qquad \nu\in\mathbb Z,
\qquad \varepsilon\in\mathcal O_{Y,\eta_E}^{\times},
\]
and define \(\ord_E(f):=\nu\), with \(\ord_E(0):=+\infty\).

Let \(I_E\) be the ideal sheaf of \(E\). For a field-valued arc
\(\widetilde\alpha\) on \(Y\), define
\[
\ord_{\widetilde\alpha}(I_E)
:=\ord_t(\widetilde\alpha^*u),
\]
where \(u\) is a local equation of \(E\) at the center; the value is
\(+\infty\) when the pulled-back ideal is zero. For \(q\geq1\), set
\[
\operatorname{Cont}_Y^q(E)
:=\{\widetilde\alpha\in\operatorname{Arc}(Y):
\ord_{\widetilde\alpha}(I_E)=q\}.
\]
The maximal divisorial set associated with \(q\ord_E\) is
\[
W_X(E,q):=
\left(
\overline{\pi_{\mathrm{arc}}(\operatorname{Cont}_Y^q(E))}
 ^{\,\operatorname{Arc}(X)}
\right)_{\mathrm{red}}
\]
\cite[Definition~2.8]{DEI}. We abbreviate \(W_X(E,q)\) to
\(W(E,q)\) when the ambient surface is clear.

Number the exceptional curves as in
Lemma~\ref{lem:e8-resolution} below and set
\begin{equation}\label{eq:e8-Ci-definition}
C_i:=W_X(E_i,1)\subset\mathcal A
\qquad(i=1,\ldots,8).
\end{equation}
Thus \(C_i\) is a maximal divisorial set, whereas \(Z_i\) is the
higher-coordinate-order cylinder defined above. The inclusion follows
from \(\pi(E_i)=\{0\}\).

For every field extension \(F/\mathbb C\), set
\[
Y_F:=Y\times_{\mathbb C}\operatorname{Spec}F,
\qquad
E_{i,F}:=E_i\times_{\mathbb C}\operatorname{Spec}F,
\]
and let \(\pi_F:Y_F\to X_F\) denote the base change of \(\pi\).

\subsection{The minimal resolution and the eight order vectors}

\begin{lemma}[Resolution graph and total transforms]
\label{lem:e8-resolution}
For $f\in\{x,y,z\}$, let $D_f:=V_X(f)$ be the corresponding
effective principal coordinate divisor, and denote its strict
transform on $Y$ by
\[
\widetilde D_f
:=\overline{\pi^{-1}\bigl(D_f\setminus\{0\}\bigr)}^{\,Y},
\]
where the closure has its reduced induced structure. Thus the symbols
$\widetilde D_x,\widetilde D_y,\widetilde D_z$ denote nonexceptional
strict transforms, not additional exceptional components.

The edges in the exceptional graph of the minimal resolution are
\[
1\! -\! 8,\quad1\! -\! 7,\quad7\! -\! 6,\quad
1\! -\! 5,\quad5\! -\! 4,\quad4\! -\! 3,\quad3\! -\! 2,
\]
and $E_k^2=-2$ for every $k$. After adjoining the strict transforms
of the three coordinate curves, the reduced total transform is an SNC
divisor with augmented graph
\[
\widetilde D_x-E_8-E_1-E_5-E_4-E_3-E_2-\widetilde D_z,
\qquad
E_1-E_7-E_6-\widetilde D_y.
\]
Here an edge means that the corresponding curves meet transversely at
one point; absence of an edge means that they are disjoint.
The corresponding exceptional order vectors are precisely
$v_1,\ldots,v_8$ listed above.
\end{lemma}

\begin{proof}
Let $G\subset\mathrm{SL}_2(\mathbb C)$ be the binary icosahedral
group, let $A=\mathbb C[u,v]$, and let $\mathfrak m=(u,v)$. The
classical invariant-theoretic presentation gives
\[
A^G\simeq
\mathbb C[x,y,z]/(x^2+y^3-z^5),
\qquad \deg(x,y,z)=(30,20,12),
\]
after rescaling the generators by nonzero constants. Blow up the
origin of $\mathbb A^2$ and take the quotient by $G$. The resulting
partial resolution $B/G$, where $B=\operatorname{Bl}_0\mathbb A^2$,
has exactly three cyclic quotient singularities; the quotient
construction and the three special orbits are described in
\cite[Sections~3.1--3.2]{PePereiraQuotient}.

For completeness, we identify this quotient with a normalized
weighted blowup. Consider the Rees algebra
\[
\mathcal R=\bigoplus_{n\geq0}\mathfrak m^nT^n
 \simeq \mathbb C[u,v,U,V]/(uV-vU),
\]
where $U=uT$ and $V=vT$. The ring on the right is an integral
three-dimensional hypersurface. Its Jacobian ideal is
$(V,-U,-v,u)$, hence its singular locus is the origin, of codimension
three. The conditions $(R_1)$ and $(S_2)$ show that $\mathcal R$ is
normal. The standard linearized finite-group quotient of
$\operatorname{Proj}$ gives
\[
B/G\simeq\operatorname{Proj}_{A^G}\mathcal R^G,
\]
and the invariant ring $\mathcal R^G$ is normal.

Pass to the second Veronese subring, regrading its component of old
degree $2q$ as degree $q$ and denoting the old symbol $T^{2q}$ by
$S^q$. Then
\[
(\mathcal R^G)^{(2)}
=\bigoplus_{q\geq0}(\mathfrak m^{2q})^G S^q
=\bigoplus_{q\geq0}I_qS^q.
\]
Indeed, if $f\in A^G$ is decomposed by ordinary degree as
$f=\sum_df_d$, then
\[
\begin{aligned}
f\in(\mathfrak m^{2q})^G
&\Longleftrightarrow f_d=0\quad(d<2q)\\
&\Longleftrightarrow
\min\{d/2:f_d\ne0\}\geq q\\
&\Longleftrightarrow
\operatorname{wt}_{(15,10,6)}(f)\geq q.
\end{aligned}
\]
Here $\min\emptyset=\operatorname{wt}(0)=+\infty$. The three
homogeneous generators have ordinary degrees $30,20,12$, and their
homogeneous relation has degree $60$; dividing ordinary degree by two
therefore gives the weight $(15,10,6)$. Consequently,
\[
I_q=\{f\in A^G:\operatorname{wt}_{(15,10,6)}(f)\geq q\}.
\]
The second Veronese subring is the invariant ring for the grading
action of $\mu_2$, so it is normal. Since taking a Veronese subring
does not alter $\operatorname{Proj}$, $B/G$ is the normalization of
the weighted blowup defined by this filtration.

At a general direction of the exceptional projective line the
stabilizer is $\{\pm I\}$. In blowup coordinates, $-I$ acts by
\[
(s,w)\longmapsto(-s,w).
\]
The quotient normal parameter is therefore $r=s^2$, and the
ramification index is two. The three homogeneous generators of
degrees $30,20,12$ vanish to these same orders along the exceptional
curve of the blowup. Their divisor orders along the central
exceptional curve $E_1$ of the quotient are consequently
\[
v_1=(15,10,6).
\]

We now determine the three arms and the positions of the strict
transforms of the coordinate curves. The three special orbits on the
projective line have effective stabilizer orders $m=2,3,5$. Near a
corresponding special direction, a generator of the full stabilizer
may be written as $\operatorname{diag}(\zeta,\zeta^{-1})$, where
$\zeta$ has order $2m$. In blowup coordinates $u=s$, $v=sw$, it acts
as
\[
(s,w)\longmapsto(\zeta s,\zeta^{-2}w).
\]
After quotienting by the central subgroup $\{\pm I\}$ and putting
$r=s^2$, the induced order-$m$ action is
\[
(r,w)\longmapsto(\epsilon r,\epsilon^{-1}w),
\qquad \epsilon=\zeta^2.
\]
Thus the local quotient is of type
$\frac1m(1,-1)=A_{m-1}$, whose minimal resolution is a chain of
$m-1$ curves of self-intersection $-2$. The strict transforms of the
axes $r=0$ and $w=0$ meet the two opposite ends of this chain.

One may see all these assertions directly in local coordinates. Take
\[
\mathsf A=r^m,\qquad \mathsf B=w^m,\qquad \mathsf C=rw,
\qquad \mathsf A\mathsf B=\mathsf C^m.
\]
The $m$ standard affine charts of the resolution of $A_{m-1}$ have
coordinates $(a_j,b_j)$, $j=0,\ldots,m-1$, and map by
\begin{equation}\label{eq:e8-cyclic-charts}
\mathsf A=a_j^{m-j}b_j^{m-j-1},\qquad
\mathsf B=a_j^jb_j^{j+1},\qquad \mathsf C=a_jb_j.
\end{equation}
These formulas satisfy $\mathsf A\mathsf B=\mathsf C^m$. Adjacent
charts are glued on $a_j\ne0$ by
\[
a_{j+1}=a_j^2b_j,\qquad b_{j+1}=a_j^{-1},
\]
and substitution leaves all three expressions in
\eqref{eq:e8-cyclic-charts} unchanged. The $(j+1)$-st component is
obtained by gluing $b_j=0$ in the $j$-th chart to $a_{j+1}=0$ in the
$(j+1)$-st chart, for $j=0,\ldots,m-2$. Its curve coordinate is glued
by $b_{j+1}=1/a_j$, and the transition
$a_{j+1}=a_j^2b_j$ for the normal coordinate gives normal bundle
$\mathcal O_{\mathbb P^1}(-2)$. In the zeroth chart the strict
transform of $r=0$ is $a_0=0$ and meets the first component $b_0=0$
transversely. In the last chart the strict transform of $w=0$ is
$b_{m-1}=0$ and meets the last component $a_{m-1}=0$ transversely.
In each intermediate chart, the adjacent chain components are the two
coordinate axes; in particular, there are no triple points.

In the notation $x=E$, $y=F$, $z=V$ of Pe Pereira, the vanishing
directions of the three basic invariants are respectively the edge
midpoint, face center, and vertex orbits, of cardinalities $30,20,12$,
and each linear factor occurs once
\cite[Section~3.1 and Figure~1]{PePereiraQuotient}. Their effective
stabilizer orders are $2,3,5$. Hence the strict transforms of the
coordinate curves $V(x),V(y),V(z)$ meet, respectively, the ends away
from the central curve of the $A_1,A_2,A_4$ chains.

It remains to compute the self-intersection of the central curve. Let
$E_B\subset B$ be the exceptional curve of the blowup, let
$\bar E\subset B/G$ be its image, and let $q:B\to B/G$ be the quotient
map. We have $E_B^2=-1$, $\deg q=|G|=120$, and the generic
ramification index along $E_B$ is two, so $q^*\bar E=2E_B$. The curve
$\bar E$ is $\mathbb Q$-Cartier: in the local $A_{m-1}$ chart,
$\mathsf A=r^m$ has divisor $m\bar E$. Applying the projection formula
for intersections after taking a Cartier multiple yields
\[
120\bar E^2=(q^*\bar E)^2=4E_B^2=-4,
\qquad \bar E^2=-\frac1{30}.
\]
Let $\rho:Y\to B/G$ resolve the three cyclic quotient points, and on
the arm corresponding to $m$ write the components from the central curve
outward as $F_{m,1},\ldots,F_{m,m-1}$. Put
\[
\rho^*\bar E=E_1+\sum_{m,j}c_{m,j}F_{m,j}.
\]
The pullback has intersection zero with every curve contracted by
$\rho$, and therefore
\[
2c_{m,j}=c_{m,j-1}+c_{m,j+1},\qquad
c_{m,0}=1,\quad c_{m,m}=0.
\]
Thus $c_{m,j}=(m-j)/m$. Intersecting with $E_1$ and applying the
projection formula gives
\[
\begin{aligned}
\bar E^2=(\rho^*\bar E)\cdot E_1
&=E_1^2+\frac12+\frac23+\frac45,\\
E_1^2&=-\frac1{30}-\frac{15+20+24}{30}=-2.
\end{aligned}
\]
The central curve is
$\mathbb P^1/(G/\{\pm I\})\simeq\mathbb P^1$, while the three arms
are smooth rational $(-2)$-chains by
\eqref{eq:e8-cyclic-charts}. Hence this resolution contains no
exceptional $(-1)$-curve and is minimal.

Let
\[
e_x=(1,0,0),\qquad e_y=(0,1,0),\qquad e_z=(0,0,1)
\]
be the multiplicity vectors attached to the three coordinate-curve
arrows. On each $(-2)$-chain, the intersection of a principal divisor
with an intermediate component is zero, giving the recurrence
\[
2v_k=v_{\mathrm{left}}+v_{\mathrm{right}}.
\]
On an arm of length $m-1$, write, from the central curve to its arrow,
\[
w_0=v_1,\ w_1,\ldots,w_{m-1},\ w_m=e_f.
\]
The recurrence says that
$w_{j+1}-w_j=w_j-w_{j-1}$, so the difference is constant and
\[
w_j=v_1+\frac jm(e_f-v_1)
=\frac{(m-j)v_1+je_f}{m}\qquad(0\leq j\leq m).
\]
For $(m,f)=(2,x),(3,y),(5,z)$ this gives, respectively,
\[
v_8=\frac{v_1+e_x}{2}=(8,5,3),
\]
\[
v_7=\frac{2v_1+e_y}{3}=(10,7,4),
\qquad
v_6=\frac{v_1+2e_y}{3}=(5,4,2),
\]
and
\[
\begin{aligned}
v_5&=\frac{4v_1+e_z}{5}=(12,8,5),&
v_4&=\frac{3v_1+2e_z}{5}=(9,6,4),\\
v_3&=\frac{2v_1+3e_z}{5}=(6,4,3),&
v_2&=\frac{v_1+4e_z}{5}=(3,2,2).
\end{aligned}
\]

On setting $x,y,z$ equal to zero in turn, the coordinate divisors are
defined respectively by $y^3-z^5$, $x^2-z^5$, and $x^2+y^3$. These
binomials have coprime exponent pairs, so the coordinate curves are
reduced and irreducible;
their strict transforms are prime divisors and occur with coefficient
one in the corresponding total transforms. Hence, for $f=x,y,z$,
we may write
\[
\operatorname{div}_Y(f\circ\pi)
=\widetilde D_f+\sum_k\ord_{E_k}(f)E_k.
\]
Its intersection with $E_k$ is zero, whence
\[
\widetilde D_f\cdot E_k
=2\ord_{E_k}(f)-\sum_{\ell\sim k}\ord_{E_\ell}(f).
\]
If $s_k=\sum_{\ell\sim k}v_\ell$, the calculation is
\[
\begin{array}{c|c|c|c}
k&2v_k&s_k&2v_k-s_k\\ \midrule
1&(30,20,12)&(30,20,12)&(0,0,0)\\
2&(6,4,4)&(6,4,3)&(0,0,1)\\
3&(12,8,6)&(12,8,6)&(0,0,0)\\
4&(18,12,8)&(18,12,8)&(0,0,0)\\
5&(24,16,10)&(24,16,10)&(0,0,0)\\
6&(10,8,4)&(10,7,4)&(0,1,0)\\
7&(20,14,8)&(20,14,8)&(0,0,0)\\
8&(16,10,6)&(15,10,6)&(1,0,0).
\end{array}
\]
The three coordinates in the last column are
$\widetilde D_x\cdot E_k$, $\widetilde D_y\cdot E_k$, and
$\widetilde D_z\cdot E_k$. Thus
\[
\widetilde D_x\cdot E_k=\delta_{k8},\qquad
\widetilde D_y\cdot E_k=\delta_{k6},\qquad
\widetilde D_z\cdot E_k=\delta_{k2}.
\]
Intersection multiplicity one shows that each strict transform is smooth and
transverse at the indicated endpoint. The other intersection numbers
are zero and rule out passage through an intersection of exceptional
curves. This proves the asserted augmented SNC graph.
\end{proof}

\begin{lemma}[Coordinate orders on maximal divisorial sets]
\label{lem:e8-generic-orders}
For every $i=1,\ldots,8$,
\[
C_i\subset Z_i,
\]
the set \(C_i\) is irreducible, and the coordinate-order vector of
its generic arc is
exactly $v_i$. Thus, if $\eta_i$ denotes the generic point of $C_i$,
then
\[
\bigl(\ord_{\eta_i}(x),\ord_{\eta_i}(y),\ord_{\eta_i}(z)\bigr)=v_i.
\]
\end{lemma}

\begin{proof}
Set
\[
E_i^\circ:=E_i\setminus\bigcup_{k\ne i}E_k
\]
and let
$\tau_{\infty,0}^Y:\operatorname{Arc}(Y)\to Y$ be the center map.
Define
\[
\mathcal T_i
:=\operatorname{Cont}^1_Y(E_i)
\cap(\tau_{\infty,0}^Y)^{-1}(E_i^\circ).
\]
Thus $\mathcal T_i$ consists exactly of the arcs centered on
$E_i^\circ$ and having contact order one with $E_i$. Since $Y$ is
smooth and $E_i^\circ$ is an irreducible smooth curve,
$\mathcal T_i$ is a nonempty open subcylinder of the irreducible
cylinder $\operatorname{Cont}^1_Y(E_i)$ and is therefore dense in it.
By continuity of $\pi_{\mathrm{arc}}$ and the definition of $W(E_i,1)$,
\[
C_i=\overline{\pi_{\mathrm{arc}}(\mathcal T_i)}^{\,\mathrm{Zar}}.
\]
Since \(\mathcal T_i\) is irreducible, its image and the closure of
that image are irreducible; hence \(C_i\) is irreducible.

For $f=x,y,z$, Lemma~\ref{lem:e8-resolution} gives the effective total
transform
\[
\operatorname{div}_Y(f\circ\pi)
=\widetilde D_f+\sum_k\ord_{E_k}(f)E_k.
\]
Take $\widetilde\alpha\in\mathcal T_i$, put
$L=\kappa(\widetilde\alpha)$, and regard it as an arc
\[
\widetilde\alpha:\Spec L[[t]]\longrightarrow Y.
\]
Let
\[
\alpha_L:=\pi\circ\widetilde\alpha:\Spec L[[t]]\longrightarrow X,
\qquad
\alpha:=\pi_{\mathrm{arc}}(\widetilde\alpha)
\in\operatorname{Arc}(X).
\]
The residue-field embedding $\kappa(\alpha)\hookrightarrow L$ extends
the residue-field arc represented by $\alpha$ to $\alpha_L$; field
extension does not change $t$-adic order. Near the center of the lift,
let $u=0$ be a local equation for $E_i$. Effectivity of the total
transform gives, for $f=x,y,z$,
\[
f\circ\pi=u^{\ord_{E_i}(f)}g_f,
\qquad g_f\in\mathcal O_{Y,\widetilde\alpha(0)}.
\]
Since $\widetilde\alpha\in\mathcal T_i$, one has
$\ord_t\widetilde\alpha^*u=1$, while the pullback of the regular
function $g_f$ has order in $\mathbb Z_{\geq0}\cup\{+\infty\}$. Hence
\[
\ord_\alpha(f)
=\ord_{\alpha_L}(f)
=\ord_t\widetilde\alpha^*(f\circ\pi)
\geq\ord_{E_i}(f)
\qquad(f=x,y,z).
\]
Thus $\pi_{\mathrm{arc}}(\mathcal T_i)\subset Z_i$. Since $Z_i$ is closed,
taking Zariski closures gives $C_i\subset Z_i$.

Now put
\[
U_i:=E_i^\circ\setminus
\bigl(\widetilde D_x\cup\widetilde D_y\cup\widetilde D_z\bigr).
\]
The augmented SNC graph in Lemma~\ref{lem:e8-resolution} shows that
$U_i$ is a nonempty Zariski-open subset of $E_i^\circ$. If
\[
c_i:\mathcal T_i\longrightarrow E_i^\circ,
\qquad \widetilde\alpha\longmapsto\widetilde\alpha(0)
\]
is the center map, then $c_i^{-1}(U_i)$ is a nonempty dense open
cylinder in $\mathcal T_i$. By continuity,
$\pi_{\mathrm{arc}}(c_i^{-1}(U_i))$ is dense in
$C_i=\overline{\pi_{\mathrm{arc}}(\mathcal T_i)}$.

Near any point of $U_i$, taking $u=0$ as a local equation of $E_i$,
we have
\[
f\circ\pi=u^{\ord_{E_i}(f)}\varepsilon_f,
\qquad \varepsilon_f\in\mathcal O_Y^\times
\]
for all $f=x,y,z$. On $c_i^{-1}(U_i)$, transversality gives
$\ord_t\widetilde\alpha^*u=1$, and
$\widetilde\alpha^*\varepsilon_f$ is a unit. Consequently,
\[
\ord_\alpha(f)=\ord_{E_i}(f).
\]
For each $f=x,y,z$, the condition
\[
\ord_\alpha(f)\geq\ord_{E_i}(f)+1
\]
on $\alpha\in\operatorname{Arc}(X)$ defines a closed cylinder.
The dense family just constructed avoids each of these three closed
cylinders; therefore the generic point of $C_i$ belongs to none of
them. Together with the lower bounds already proved, this yields
\[
\bigl(\ord_{\eta_i}(x),\ord_{\eta_i}(y),\ord_{\eta_i}(z)\bigr)
=\bigl(\ord_{E_i}(x),\ord_{E_i}(y),\ord_{E_i}(z)\bigr)=v_i.
\]

We shall also need the relation with Nash arc families. Let
\[
\mathcal H_i=(\tau_{\infty,0}^Y)^{-1}(E_i),\qquad
N_i=\{\alpha\in\mathcal A:\alpha\ne0_\infty,
\ \widetilde\alpha(0)\in E_i\}.
\]
Since $Y$ is smooth, for every $m$ the inverse image of $E_i$ under
$\tau_{m,0}^Y:J_m(Y)\to Y$ is locally an affine-space bundle over
$E_i$; consequently $\mathcal H_i$ is irreducible
\cite[Corollary~2.11]{EM}. The subset $\mathcal T_i$ is a
nonempty open cylinder in $\mathcal H_i$:
its center avoids the other exceptional components and the first
coefficient of a local equation of $E_i$ is nonzero. Hence
$\mathcal H_i=\overline{\mathcal T_i}^{\,\operatorname{Arc}(Y)}$
and continuity
gives
\[
\pi_{\mathrm{arc}}(\mathcal H_i)
\subset\overline{\pi_{\mathrm{arc}}(\mathcal T_i)}=C_i.
\]
Conversely,
$\pi_{\mathrm{arc}}(\mathcal T_i)\subset N_i\subset
\pi_{\mathrm{arc}}(\mathcal H_i)$, so
\begin{equation}\label{eq:e8-nash-identification}
C_i=\overline{N_i}^{\,\mathrm{Zar}}.
\end{equation}
After any extension of the base field, $Y_F$ remains smooth and
$E_{i,F}\simeq\mathbb P^1_F$ remains geometrically irreducible. The
nonempty open cylinders above remain nonempty. The same proof thus
applies to $Y_F$ and $E_{i,F}$.
\end{proof}

\subsection{Exact order implies transversality}

\begin{lemma}[Transversality]
\label{lem:e8-transversality}
Let $F/\mathbb C$ be any field extension and let
$\alpha:\Spec F[[t]]\to X_F$ belong to $W_i(F)$. Its unique lift
$\widetilde\alpha:\Spec F[[t]]\to Y_F$ meets $E_{i,F}$ transversely at
a smooth point of the total exceptional divisor. Equivalently,
\[
\widetilde\alpha(0)\in
E_{i,F}\setminus\bigcup_{k\ne i}E_{k,F},
\qquad
\ord_t\widetilde\alpha^*I_{E_{i,F}}=1.
\]
For every further field extension $L/F$, the lift, the exceptional
component containing its center, and this contact order are compatible
with coefficientwise extension of scalars.
\end{lemma}

\begin{proof}
Smoothness is preserved by base change, so $Y_F$ is a smooth surface.
Properness, separatedness, and the property of being an isomorphism
are also preserved by base change; hence
$\pi_F:Y_F\to X_F$ is proper and separated and is an isomorphism away
from the origin. All three coordinate orders of $\alpha$ are finite,
so the generic point of $\alpha$ lies in $X_F\setminus\{0\}$ and has a
unique lift there. Because $F[[t]]$ is a discrete valuation ring, the
valuative criterion for properness extends this lift to
$\Spec F[[t]]$; separatedness makes the extension unique.

Put $P_F=\widetilde\alpha(0)$. This is an $F$-rational closed point of
$Y_F$, although its image in $Y$ need not be closed. We therefore make
all local calculations on $Y_F$. The local ring
$\mathcal O_{Y_F,P_F}$ is a two-dimensional regular local ring with
residue field $F$. The augmented boundary of
Lemma~\ref{lem:e8-resolution} remains SNC after base change: every
component remains smooth, every indicated intersection remains
transverse, and no new intersections appear. Thus one may choose
regular parameters $u,v$ in this local ring so that the boundary
components through $P_F$ are either $u=0$, or $u=0$ and $v=0$.
The effective total-transform identities are preserved by base
change, and therefore, for $f=x,y,z$,
\[
f\circ\pi_F=u^{A_f}v^{B_f}\varepsilon_f,
\qquad \varepsilon_f\in\mathcal O_{Y_F,P_F}^{\times},
\]
where $B_f=0$ if only one boundary component passes through $P_F$.
The local homomorphism to $F[[t]]$ sends units to units. The pullback
of every local equation of a boundary component has finite positive
order because none of the three coordinate series is identically
zero. Write
\[
a=\ord_t\widetilde\alpha^*u,
\]
and, if there is a second boundary component, write
$b=\ord_t\widetilde\alpha^*v$; otherwise set $b=0$. Then
\[
\ord_\alpha(f)=aA_f+bB_f.
\]

The augmented graph now shows that the order vector of $\alpha$ must
have one of the following three forms. In the first form $a$ is a
positive integer, and in the other two both $a$ and $b$ are positive
integers:
\[
av_k,
\qquad
av_k+bv_\ell\quad(E_k\cap E_\ell\ne\emptyset),
\]
or
\[
av_8+b(1,0,0),\qquad
av_6+b(0,1,0),\qquad
av_2+b(0,0,1).
\]
The vectors $v_1,\ldots,v_8$ are pairwise distinct and primitive:
\[
\gcd(p_i,q_i,r_i)=1\qquad(i=1,\ldots,8).
\]
If $av_k=v_i$, the gcd of the three coordinates gives $a=1$, and
then pairwise distinctness gives $k=i$.

Suppose next that the center lies at an intersection
$E_k\cap E_\ell$. Since $a,b\geq1$,
\[
av_k+bv_\ell\geq v_k+v_\ell
\]
coordinatewise. With the exception of the edge $E_2-E_3$, the sums of
the endpoint vectors are
\[
\begin{array}{c|c}
\text{edge}&v_k+v_\ell\\ \hline
E_1-E_8&(23,15,9)\\
E_1-E_7&(25,17,10)\\
E_7-E_6&(15,11,6)\\
E_1-E_5&(27,18,11)\\
E_5-E_4&(21,14,9)\\
E_4-E_3&(15,10,7).
\end{array}
\]
Each of these sums has at least one coordinate strictly larger than
the corresponding coordinate of $v_1=(15,10,6)$, while $v_1$
dominates the other seven vectors coordinatewise. None can therefore
equal a $v_i$. On the remaining edge $E_2-E_3$,
\[
av_3+bv_2=(6a+3b,4a+2b,3a+2b).
\]
For $a,b\geq1$ this vector dominates
$v_2+v_3=(9,6,5)$. Among the eight vectors, only
$v_5=(12,8,5)$ and $v_1=(15,10,6)$ dominate $(9,6,5)$. Equality with
$v_5$ would give $3a+2b=5$, hence $a=b=1$, but then the first
coordinate is $9$, not $12$. Equality with $v_1$ would give
$3a+2b=6$ and hence $a\leq1$; substituting $a=1$ gives
$b=3/2$, contrary to $b\in\mathbb Z_{>0}$. Thus no intersection of
exceptional components yields any $v_i$.

At the three arrow points the respective order vectors are
\[
(8a+b,5a,3a),\qquad
(5a,4a+b,2a),\qquad
(3a,2a,2a+b).
\]
The two equations not containing $b$ first determine $i$ and the
positive integer $a$; the remaining coordinate then determines $b$.
The complete list of candidates is
\[
\begin{array}{c|c|c|r}
\text{arrow}&\text{equations not containing $b$}&(i,a)&b\\ \midrule
\widetilde D_x&5a=q_i,\ 3a=r_i&(1,2)&15-8\cdot2=-1\\
&& (8,1)&8-8=0\\ \midrule
\widetilde D_y&5a=p_i,\ 2a=r_i&(1,3)&10-4\cdot3=-2\\
&& (6,1)&4-4=0\\
&& (7,2)&7-4\cdot2=-1\\ \midrule
\widetilde D_z&3a=p_i,\ 2a=q_i&(1,5)&6-2\cdot5=-4\\
&& (2,1)&2-2=0\\
&& (3,2)&3-2\cdot2=-1\\
&& (4,3)&4-2\cdot3=-2\\
&& (5,4)&5-2\cdot4=-3.
\end{array}
\]
This exhausts the coordinate comparisons with all eight vectors, and
every candidate has $b\leq0$, contradicting $b\geq1$ at an arrow
point. The only possible case is therefore $a=1$ and $k=i$ in the
first form. This proves that $P_F$ lies on no other exceptional
component and that the contact order with $E_{i,F}$ is one.

Finally, let $L/F$ be any further extension, with embedding
$\iota:F\hookrightarrow L$. The coefficientwise embedding
$F[[t]]\hookrightarrow L[[t]]$ preserves the first nonzero
coefficient, and hence
\[
\ord_t\Bigl(\sum c_nt^n\Bigr)
=\ord_t\Bigl(\sum\iota(c_n)t^n\Bigr);
\]
the zero series has order $+\infty$ on both sides. Composing
$\widetilde\alpha$ with
$\Spec L[[t]]\to\Spec F[[t]]$ and then mapping to $Y_L$ produces a
lift of $\alpha_L$; uniqueness identifies it with
$\widetilde{\alpha_L}$. Whether the center lies on, or avoids, a
given boundary component is determined by the vanishing, or
nonvanishing, of the constant term of a local equation, and both
properties are preserved by a field embedding. The same displayed
order identity preserves the contact order. Throughout this extension
the arc parameter $t$ remains fixed. Applying the argument with
$F=\kappa(\alpha)$ proves compatibility after every residue-field
extension; neither $F$ nor $L$ need be algebraically closed.
\end{proof}
\subsection{Extraction of a generically stable component}

Let
$\tau_{\infty,m}^X:\operatorname{Arc}(X)\to J_m(X)$ be truncation.
Following de Fernex--Ein--Ishii, a subset of the form
$(\tau_{\infty,m}^X)^{-1}(\Sigma)$, with
$\Sigma\subset J_m(X)$ constructible, is called a cylinder. Set
\[
U:=\operatorname{Arc}(X)\setminus\operatorname{Arc}(\Sing X).
\]
For a closed quasi-cylinder
$C\nsubseteq\operatorname{Arc}(\Sing X)$,
\cite[Definition~3.2]{DEI} supplies a cylinder $Q$ such that
\[
C\cap U=\overline Q\cap U.
\]
Here and throughout this subsection, closures are Zariski closures
in $\operatorname{Arc}(X)$. In particular, this condition involves
the closure of $Q$, rather than $Q$ itself. The following lemma
establishes the finite-level description needed for the eight
exact-order strata of the surface $X=V(x^2+y^3-z^5)$.

\begin{lemma}[Generic stability on exact-order strata]
\label{lem:e8-stable-component}
Let $p,q,r$ be positive integers satisfying
\[
\max(q,r)<p<\min(2q,4r),
\]
and let $W=W(p,q,r)$ be the reduced locally closed stratum defined by
\[
\ord_t x=p,\qquad \ord_t y=q,\qquad \ord_t z=r.
\]
Suppose that $C\subset\operatorname{Arc}(X)$ is closed and that
$C\cap U=\overline Q\cap U$ for a cylinder $Q$ defined at jet level
$m\geq0$. Then $C\cap W$ is a cylinder defined at level $m+p$ and
has finitely many irreducible components. For every nonempty
irreducible component $T$ of $C\cap W$, the closure
$D:=\overline T$ is generically stable in the sense of Reguera's
corrected Definition~3.1.
\end{lemma}

\begin{proof}
\emph{The coefficient ring of the stratum.}
Write $x(t)=\sum_{a\geq0}x_at^a$, and similarly for $y(t)$ and
$z(t)$. Set all coefficients below $p,q,r$, respectively, equal
to zero and put $h=x_py_qz_r$. Define
\[
P_0=\mathbb C[x_p,y_q,\ldots,y_p,z_r,\ldots,z_p,h^{-1}].
\]
Let $F_n$ denote the coefficient of $t^n$ in $x(t)^2+y(t)^3-z(t)^5$
after these lower coefficients have been set to zero. For $n\leq2p$,
the polynomial $F_n$ belongs to $P_0$: a term from $y(t)^3$ uses
indices at most $n-2q<p$, a term from $z(t)^5$ uses indices at
most $n-4r<p$, and a term from $x(t)^2$ uses only $x_p$.
Set
\[
B_0=P_0/(F_0,\ldots,F_{2p}),\qquad B=(B_0)_{\mathrm{red}}.
\]
For each $k>p$, the next coefficient equation has the form
\begin{equation}
\label{eq:e8-stability-triangular}
F_{p+k}=2x_px_k+G_k,
\end{equation}
where $G_k$ uses only $x$-coefficients of index less than $k$,
$y$-coefficients of index at most $k+p-2q<k$, and $z$-coefficients
of index at most $k+p-4r<k$. Since $2x_p$ is a unit, these equations
recursively eliminate all $x_k$ with $k>p$. They give mutually
inverse ring homomorphisms between the localized coefficient ring
of the stratum before reduction and
$B_0[y_k,z_k\mid k>p]$: each polynomial involves only finitely many
coefficients, so the inverse identities follow by finite induction
from \eqref{eq:e8-stability-triangular}. The nilradical of a
polynomial ring is the extension of the nilradical of its
coefficient ring, also for infinitely many variables. Taking
reductions therefore gives
\begin{equation}
\label{eq:e8-stability-ring}
R_W:=\Gamma(W,\mathcal O_W)
\simeq B[y_k,z_k\mid k>p].
\end{equation}
In particular, $B$ is a reduced Noetherian $\mathbb C$-algebra.

\emph{Automorphisms fixing a finite jet level.}
On $\mathbb C[x,y,z]$ consider the derivations
\[
\delta_y=2x\partial_y-3y^2\partial_x,
\qquad
\delta_z=2x\partial_z+5z^4\partial_x.
\]
Both annihilate $x^2+y^3-z^5$, and hence descend to
$\mathcal O(X)$. For $N\geq1$ and $\lambda\in\mathbb C$, the
$t$-adically convergent exponential
\[
\exp(\lambda t^N\delta)
=\sum_{\ell\geq0}\frac{\lambda^\ell t^{N\ell}}{\ell!}\delta^\ell,
\qquad \delta\in\{\delta_y,\delta_z\},
\]
defines a continuous automorphism of $\mathcal O(X)[[t]]$ fixing
$t$, with inverse $\exp(-\lambda t^N\delta)$. Multiplicativity
follows from the iterated Leibniz rule, and the inverse identity
follows by multiplying the two convergent exponential series.
No local nilpotence of $\delta$ on $\mathcal O(X)$ is required.
More explicitly, for every $\mathbb C$-algebra $R$ and every arc
$\alpha:\mathcal O(X)\to R[[t]]$, the formula
\begin{equation}
\label{eq:e8-stability-flow}
f\longmapsto
\sum_{\ell\geq0}\frac{\lambda^\ell t^{N\ell}}{\ell!}
\alpha(\delta^\ell f)
\end{equation}
defines the transformed arc. Its coefficient of index $a$ uses
only $\ell\leq a/N$, and is a polynomial in finitely many input
coefficients. These formulas and their inverses are natural in
$R$ and define scheme automorphisms of $\operatorname{Arc}(X)$.
They induce automorphisms on its reduction as well.
If $N>m$, they fix every coefficient of index at most $m$; thus
they preserve $Q$ and its Zariski closure.

In the polynomial ring $\mathbb C[x,y,z]$, give $x,y,z$ weights
$p,q,r$. The derivations $\delta_y$ and
$\delta_z$ increase weighted order by at least, respectively,
\[
s_y=\min(p-q,2q-p)>0,\qquad
s_z=\min(p-r,4r-p)>0.
\]
For arcs in $W$, these bounds imply the corresponding lower
bounds on $t$-adic order. Consequently, the automorphisms
\eqref{eq:e8-stability-flow} preserve the lower coefficient
vanishings and the three leading coefficients. They restrict to
automorphisms of $W$ and fix $h$.

For $k>p$, take $N=k-p$. On coefficients of indices at most $k$,
the $\delta_y$-flow fixes all $x$-coefficients, all $z$-coefficients,
and all lower $y$-coefficients, and acts by
\[
y_k\longmapsto y_k+2x_p\lambda.
\]
Indeed, the first correction to $x$ has order at least
$N+2q>k$, whereas $\delta_y(y)=2x$. For $\ell\geq2$, the
remaining corrections to $y$ have order at least
\[
N\ell+p+(\ell-1)s_y>k.
\]
The same argument for $\delta_z$ gives
$z_k\mapsto z_k+2x_p\lambda$, fixes all lower free coefficients
and $y_k$, and has its first correction to $x$ in order
$N+4r>k$. In particular, the restrictions of these flows to
$B[y_a,z_a\mid p<a\leq k]$ are precisely the indicated
translations. Coefficients of larger index may change and are
not used in this assertion.

\emph{Descent of the defining ideal to a finite level.}
The singular locus of $X$ is the origin, so $W\subset U$. Hence
\[
C\cap W=\overline Q\cap W.
\]
Let $J\subset R_W$ be the radical ideal of this closed subset
and put $K=m+p$. For every $k>K$, the preceding flows have
$N=k-p>m$ and preserve both $W$ and $\overline Q$.
Their pullbacks therefore preserve $J$. Set
\[
B_K=B[y_a,z_a\mid p<a\leq K].
\]
We claim that
\begin{equation}
\label{eq:e8-stability-descent}
J=(J\cap B_K)R_W.
\end{equation}
Let $f\in J$ and suppose that its largest free coefficient index
is $k>K$. Write
\[
f=\sum_{a=0}^d f_a y_k^a,
\qquad
f_a\in B[y_b,z_b\mid p<b<k][z_k].
\]
The $\delta_y$-flow sends $f$ to $f(y_k+\lambda u)$, where
$u=2x_p\in B^\times$. This polynomial belongs to $J$ for every
$\lambda\in\mathbb C$. Choose $d+1$ distinct complex scalars.
The scalar Vandermonde matrix is invertible, so every coefficient
in $\lambda$ of $f(y_k+\lambda u)$ belongs to $J$. Its top
coefficient is $u^df_d$, whence $f_d\in J$. Subtracting
$f_dy_k^d$ and repeating proves $f_a\in J$ for every $a$.
Applying the $\delta_z$-flow to each $f_a$ extracts all
coefficients in $z_k$ in the same manner. Thus $f$ lies in the
extension of the intersection of $J$ with the subring of strictly
smaller free indices. Descending through the finitely many
indices in $f$ proves \eqref{eq:e8-stability-descent}.
This argument uses only units in $B$, so it does not require $B$
to be a domain or $J$ to be prime.

Since $B_K$ is Noetherian, $\mathfrak j_K:=J\cap B_K$ is finitely
generated. Its generators, represented in the displayed
coordinates of $B_K$, involve only coefficients of index at
most $K$. Clearing powers of $h$, adjoining the lower coefficient
vanishings, and retaining the condition $h\ne0$ express
$C\cap W$ as the inverse image of a constructible subset of
$J_K(X)$.
Thus $C\cap W$ is a cylinder defined at level $K=m+p$.

\emph{Irreducible components and generic stability.}
Equation~\eqref{eq:e8-stability-descent} gives
\[
R_W/J\simeq(B_K/\mathfrak j_K)[y_a,z_a\mid a>K].
\]
For an arbitrary ring $A_0$, the minimal prime ideals of a
polynomial ring over $A_0$ are exactly the extensions of its
minimal prime ideals, even for an infinite set of variables.
Indeed, the extension of a prime is prime; if $P$ is a minimal
prime of the polynomial ring, choose a minimal prime of $A_0$
contained in its contraction and use minimality. Conversely,
contraction and extension show that an extended minimal prime
contains no strictly smaller prime. Applying this observation
to the Noetherian ring $B_K/\mathfrak j_K$ proves that $C\cap W$ has
finitely many irreducible components and that each component
$T$ is defined in $R_W$ by the extension of a finitely generated
prime ideal of $B_K$.

Put
\[
H:=D_{\operatorname{Arc}(X)}(h).
\]
This is an affine open subset of the entire arc scheme, since
$X$ is affine. The stratum $W$ is the reduced closed subscheme
of $H$ cut out by the finitely many lower coefficient
vanishings, including the centering conditions. Lift the finite
generators defining $T$ in $R_W$ to $\Gamma(H,\mathcal O_H)$.
Together with these lower coefficient functions, they generate
a finitely generated ideal $I_T$ such that
\[
I_H(T)=\sqrt{I_T}.
\]
Reduction causes no additional finiteness requirement here:
if $L$ is the ideal of the lower coefficient functions, then
$\sqrt{\sqrt L+(g_1,\ldots,g_s)}
=\sqrt{L+(g_1,\ldots,g_s)}$.
In particular, $T$ is closed in $H$. The open-subspace closure
identity now gives
\[
D\cap H=\overline T\cap H=\overline T^{\,H}=T.
\]
This set is nonempty, its defining ideal is the radical of a
finitely generated ideal, and it contains the generic point of
the irreducible closed set $D$. Moreover, $T\subset W\subset U$
implies $D\nsubseteq\operatorname{Arc}(\Sing X)$. Passing to
$H_{\mathrm{red}}$ preserves the same radical description and
gives the required affine open in $\operatorname{Arc}(X)_{\mathrm{red}}$.
These verify
all requirements of \cite[corrected Definition~3.1]{RegueraCorr}.
\end{proof}

For the eight order vectors of Subsection~\ref{subsec:e8-setup},
the inequalities required in Lemma~\ref{lem:e8-stable-component}
are verified by the following positive differences:
\[
\begin{array}{c|c|rrrr}
i&(p_i,q_i,r_i)&p_i-q_i&p_i-r_i&2q_i-p_i&4r_i-p_i\\ \midrule
1&(15,10,6)&5&9&5&9\\
2&(3,2,2)&1&1&1&5\\
3&(6,4,3)&2&3&2&6\\
4&(9,6,4)&3&5&3&7\\
5&(12,8,5)&4&7&4&8\\
6&(5,4,2)&1&3&3&3\\
7&(10,7,4)&3&6&4&6\\
8&(8,5,3)&3&5&2&4.
\end{array}
\]
Consequently, the lemma applies to every $W_i=W(p_i,q_i,r_i)$.

\subsection{Separation of the exact-order families}

\begin{lemma}[Separation]
\label{lem:e8-separation}
For all $i\ne j$,
\[
W_i\cap C_j=\emptyset.
\]
\end{lemma}

\begin{proof}
We first identify the only ordered pairs that are not excluded
immediately by the coordinate orders.  Lemma~\ref{lem:e8-generic-orders}
gives $C_j\subset Z_j$.  Hence, if
$\gamma\in W_i\cap C_j$, then the definition of $W_i$ and the
inclusion $\gamma\in Z_j$ give
\[
p_i=\ord_\gamma(x)\geq p_j,\qquad
q_i=\ord_\gamma(y)\geq q_j,\qquad
r_i=\ord_\gamma(z)\geq r_j.
\]
Thus
\[
W_i\cap C_j\ne\emptyset\quad\Longrightarrow\quad v_j\leq v_i.
\]
In particular, $v_j\nleq v_i$ implies
$W_i\cap C_j=\emptyset$.  Notice that this is only a necessary
condition for nonempty intersection, not a sufficient one.  With the
above convention, $v_j\leq v_i$ also gives the reverse inclusion of
order cylinders $Z_i\subseteq Z_j$.

For the eight vectors in question, direct coordinatewise comparison
gives the following strict chain in the product order:
\[
v_2<v_6<v_3<v_8<v_4<v_7<v_5<v_1.
\]
Indeed, the successive differences, in the same order, are
\[
\begin{gathered}
(2,2,0),\quad(1,0,1),\quad(2,1,0),\quad(1,1,1),\\
(1,1,0),\quad(2,1,1),\quad(3,2,1),
\end{gathered}
\]
all of which are coordinatewise nonnegative and nonzero.  Therefore,
for a fixed $i$, the possible indices $j\ne i$ are exactly those
lying strictly to the left of $i$ in this chain.  Equivalently, the
complete list of ordered pairs not yet excluded is
\[
\begin{array}{c|l}
i&\text{indices $j\ne i$ satisfying $v_j\leq v_i$}\\ \midrule
2&\emptyset\\
6&2\\
3&2,6\\
8&2,3,6\\
4&2,3,6,8\\
7&2,3,4,6,8\\
5&2,3,4,6,7,8\\
1&2,3,4,5,6,7,8.
\end{array}
\]
This table contains all $28$ remaining ordered pairs; the other $28$
ordered pairs with $i\ne j$ have already been eliminated.  It remains
to rule out the displayed pairs.  Fix one of them, so that
\[
v_j\leq v_i,\qquad v_j\ne v_i,
\]
and suppose, for a contradiction, that
$\gamma\in W_i\cap C_j$.

The maximal divisorial set $C_j=W(E_j,1)$ is a quasi-cylinder by
\cite[Theorem~3.9]{DEI}. The theorem applies to the complex variety
$X$, the resolution $\pi:Y\to X$, the smooth prime divisor $E_j$ on
$Y$, and the positive integer $q=1$; these hypotheses were verified
above and in Lemma~\ref{lem:e8-resolution}. Since $X$ has an isolated
singularity,
\[
\operatorname{Arc}(\Sing X)=\{0_\infty\},
\]
and $H_i:=D_{\operatorname{Arc}(X)}(h_i)$ does not contain the zero
arc. Since $\mathcal A$ is closed in $\operatorname{Arc}(X)$,
$C_j$ is also closed in the entire arc scheme. The quasi-cylinder
condition is used with its closure formulation above. Put
\[
S:=C_j\cap W_i=C_j\cap Z_i\cap D_{\mathcal A}(h_i).
\]
By Lemma~\ref{lem:e8-stable-component} and the verified inequalities
for $v_i$, the set $S$ has finitely many irreducible components.
Choose one such component $T$ containing $\gamma$ and set
\[
D:=\overline T^{\,\mathrm{Zar}}\subset C_j.
\]
Lemma~\ref{lem:e8-stable-component} shows that $D$ is generically
stable in the corrected sense. Moreover, $D\subset Z_i$. By
Lemma~\ref{lem:e8-generic-orders}, the generic order vector of $C_j$
is $v_j\ne v_i$; under the present assumption at least one coordinate
is strictly smaller than the corresponding coordinate of $v_i$.
Thus $C_j\nsubseteq Z_i$ and
\[
D\subsetneq C_j.
\]

We next identify the generic point of $D$. Let
\[
F:=C_j\cap Z_i,
\qquad S=F\cap D(h_i).
\]
The set $F$ is closed and contains $T$, hence $D\subset F$. Since $T$
is an irreducible component and therefore closed in $S$,
\[
T=\overline T^{\,S}
=S\cap\overline T^{\,\mathcal A}
=S\cap D
=D\cap D(h_i).
\]
Thus $T$ is a nonempty dense open subset of the irreducible closed set
$D$, and they have the same generic point
\[
\xi:=\eta_D=\eta_T\in T\subset S\subset W_i.
\]

We now verify all hypotheses of Reguera's corrected curve selection
lemma. First, $\mathbb C$ is perfect, and $X$ is an integral,
separated variety of finite type. Second, $D$ and $C_j$ are
irreducible and closed, with $D\subsetneq C_j$: the irreducibility of
$C_j$ follows from Lemma~\ref{lem:e8-generic-orders}, and $D$ is the
closure of the irreducible component $T$. Third,
Lemma~\ref{lem:e8-stable-component} supplies an affine open
neighborhood $H_i$ in $\operatorname{Arc}(X)$, with $D\cap H_i=T$,
and the radical of a finitely generated ideal required by
Reguera's corrected Definition~3.1. In particular, this is an
open neighborhood in the entire arc scheme, not merely in
$\mathcal A$. Also $\xi\in T\subset H_i$, and $H_i$ avoids the
zero arc, so
$D\nsubseteq\operatorname{Arc}(\Sing X)$. Finally, the
special point is the generic point $\xi$ of $D$, and curve selection
is applied on $\operatorname{Arc}(X)_{\mathrm{red}}$.

Accordingly, applying \cite[Corollary~4.8]{RegueraOriginal} together
with its correction in
\cite[Corollary~4.8 and Remark~C.4]{RegueraCorr} to
$D\subsetneq C_j$ produces a finite extension
$K/\kappa(\xi)$ and a morphism
\[
\Phi:\Spec K[[s]]\longrightarrow C_j
\]
such that
\[
\Phi(0)=\xi,
\qquad
\beta:=\Phi(\eta)\in C_j\setminus D.
\]
Let $\alpha_0$ be the special $K$-valued arc represented by $\Phi$;
it is obtained from the residue-field arc of $\xi$ by the extension
$\kappa(\xi)\hookrightarrow K$. Its point in the arc space is $\xi$,
and $\alpha_0\in W_i(K)$.

Because $\xi\in D(h_i)$, the reduction of $\Phi^*h_i$ modulo $(s)$ is
nonzero; hence
\[
\Phi^*h_i\in K[[s]]^\times.
\]
More explicitly, $\Phi^*h_i$ is the product of the three coefficients
$\Phi^*x_{p_i}$, $\Phi^*y_{q_i}$, and $\Phi^*z_{r_i}$. A product is a
unit in the local ring $K[[s]]$ only if every factor is a unit. These
coefficients are consequently nonzero in $K((s))$, and
\[
\ord_\beta x\leq p_i,\qquad
\ord_\beta y\leq q_i,\qquad
\ord_\beta z\leq r_i.
\]
The general arc
$\beta:\Spec(K((s))[[t]])\to X$ is not the zero arc. Since
$\Sing X=\{0\}$ and $\pi$ is an isomorphism over
$X\setminus\{0\}$, its generic point has a unique lift to $Y$; the
valuative criteria for properness and separatedness extend it
uniquely to
\[
\widetilde\beta:\Spec(K((s))[[t]])\longrightarrow Y.
\]

If $\widetilde\beta(0)\in E_i$, effectivity of the total transforms
would give
\[
\ord_\beta(f)\geq\ord_{E_i}(f)
\qquad(f=x,y,z).
\]
Combined with the opposite inequalities above, this would imply
$\ord_\beta(x,y,z)=v_i$, hence $\beta\in W_i$. In that case $\Phi$
would factor through $S$: the coefficients below orders
$p_i,q_i,r_i$ vanish in $K((s))$ and therefore already vanish in
$K[[s]]$, while $h_i$ is a unit.

Under this assumption, inside $S$, the point $\xi$ is a specialization
of $\beta$. The
irreducible closed subset $\overline{\{\beta\}}^{\,S}$ therefore
contains $\overline{\{\xi\}}^{\,S}=T$. By maximality of the
irreducible component $T$ of $S$, one has
$\overline{\{\beta\}}^{\,S}=T$, so $\beta\in T\subset D$, contrary to
$\beta\notin D$. We conclude that
\[
\widetilde\beta(0)\notin E_i.
\]

Because $\beta$ is centered at the origin, its lifted center lies in
the exceptional set. Choose an exceptional component $E_k$ containing
that center. Then $k\ne i$. On the other hand,
$\alpha_0\in W_i(K)$, so Lemma~\ref{lem:e8-transversality} shows that
the lift of the special arc $\alpha_0$ is transverse to $E_i$. The
morphism $\Phi$, equivalently the induced map
\[
\Spec K[[s,t]]\longrightarrow X,
\]
is therefore a formal wedge whose special arc is transverse to $E_i$
and whose general arc lifts with center on $E_k$.

Let $\Omega=K^{\mathrm{alg}}$ and fix an embedding
$K\hookrightarrow\Omega$. The structure morphisms give injections
\[
\mathbb C\hookrightarrow\kappa(\xi)\hookrightarrow K
\hookrightarrow\Omega.
\]
Thus $\Omega$ is an uncountable algebraically closed field of
characteristic zero. To spell out the scalar extension of the wedge,
write $R_X=\Gamma(X,\mathcal O_X)$ and let
$\varphi:R_X\to K[[s,t]]$ be the homomorphism corresponding to the
original wedge. If
$\iota:K[[s,t]]\hookrightarrow\Omega[[s,t]]$ is the coefficientwise
embedding, define
\[
R_X\otimes_{\mathbb C}\Omega\longrightarrow\Omega[[s,t]],
\qquad
r\otimes a\longmapsto a\,\iota(\varphi(r)).
\]
This gives a formal $\Omega$-wedge
\[
\Psi_\Omega:\Spec\Omega[[s,t]]\longrightarrow X_\Omega.
\]
Since \(\Phi\) takes values in
\(C_j\subset\operatorname{Arc}^0(X)\), the associated two-parameter
morphism maps the divisor \(V(t)\) to the origin. The same remains
true after scalar extension; hence
\(\Psi_\Omega(V(t))=\{0\}\), as required for a wedge in the pointed
surface \((X_\Omega,0)\).

Lemma~\ref{lem:e8-geometric-normality} shows that $X_\Omega$ is a normal
surface singular only at the origin. The surface $Y_\Omega$ is smooth,
$\pi_\Omega$ is proper and is an isomorphism off the origin, and hence
$Y_\Omega\to X_\Omega$ is a resolution. Its exceptional components
$E_{\ell,\Omega}\simeq\mathbb P^1_\Omega$ have the same dual graph and
self-intersection $-2$. Thus \(\pi_\Omega:Y_\Omega\to X_\Omega\) remains the minimal good resolution;
this is also the base-change stability established in
\cite[Section~7.1, pp.~162--163]{FdBTopology}. For a normal surface
singularity, the irreducible components of the exceptional divisor of
the minimal resolution are precisely the essential components
\cite[p.~132]{FdBTopology}. Consequently every
$E_{\ell,\Omega}$ is essential.

The scalar-extension argument at the end of the proof of
Lemma~\ref{lem:e8-transversality} also applies to the general arc over
$K((s))$. Its unique lift extends to the unique lift over
$\Omega((s))$, and membership of its center in an exceptional
component is preserved. Applied to the special arc over $K$, the same
argument also preserves transversality. Throughout, $t$ remains the
arc parameter. If $N_{E_{\ell,\Omega}}$ denotes the family of nonzero
arcs centered at the origin whose unique lift has center on
$E_{\ell,\Omega}$, and
$\dot N_{E_{\ell,\Omega}}$ denotes those whose lift is transverse at a
smooth point of the total exceptional divisor, then
\[
(\Psi_\Omega)_0\in\dot N_{E_{i,\Omega}},
\qquad
(\Psi_\Omega)_\eta\in N_{E_{k,\Omega}}.
\]
Set
\[
C_{\ell,\Omega}:=
\overline{N_{E_{\ell,\Omega}}}^{\,\mathrm{Zar}}
\subset\operatorname{Arc}^0(X_\Omega).
\]
The proof of \eqref{eq:e8-nash-identification} applies verbatim over
$\Omega$, giving
$C_{\ell,\Omega}=W_{X_\Omega}(E_{\ell,\Omega},1)$.

We use the terminology and numbering of the published version of
\cite{FdBTopology}. By
\cite[Definitions~2--3, pp.~137--138]{FdBTopology}, a formal wedge
realizes an
adjacency from $E_u$ to $E_v$ when its generic arc belongs to
$N_{E_u}$ and its special arc belongs to $\dot N_{E_v}$; it avoids a
proper closed subset $\mathcal Z\subset
\overline{N_{E_u}}^{\,\mathrm{Zar}}$ when its generic arc does not
belong to $\mathcal Z$.
Accordingly, $\Psi_\Omega$ realizes the adjacency
\[
E_{k,\Omega}\longrightarrow E_{i,\Omega}.
\]

Set
\[
\mathcal Z:=\varnothing
\subsetneq\overline{N_{E_{k,\Omega}}}^{\,\mathrm{Zar}}.
\]
The wedge $\Psi_\Omega$ automatically avoids $\mathcal Z$. Moreover,
$\Omega$ is an uncountable algebraically closed field of
characteristic zero, $(X_\Omega,0)$ is a normal surface singularity,
$E_{i,\Omega}$ is essential, and $E_{k,\Omega}\ne E_{i,\Omega}$.
The special arc is transverse to $E_{i,\Omega}$ at a smooth point of
the total exceptional divisor, and the lifted center of the generic
arc lies on $E_{k,\Omega}$. Thus $\Psi_\Omega$ satisfies
condition~(3) of Theorem~33. Consequently,
\cite[Theorem~33, pp.~163--166, implication~$(3)\Rightarrow(1)$]
{FdBTopology} gives
\[
N_{E_{i,\Omega}}\subset
\overline{N_{E_{k,\Omega}}}^{\,\mathrm{Zar}}.
\]
Since the right-hand side is closed, taking the Zariski closure of the
left-hand side gives
\[
C_{i,\Omega}\subset C_{k,\Omega}.
\]

Finally, the Nash map is bijective for algebraic surfaces over an
algebraically closed field of characteristic zero
\cite[Main Theorem]{FdBPP}. Since the exceptional components are
essential and \eqref{eq:e8-nash-identification} holds over $\Omega$,
the eight sets
$C_{\ell,\Omega}$ are the distinct Nash irreducible components of the
arc space centered at the origin. They are therefore maximal
irreducible closed subsets of $\operatorname{Arc}^0(X_\Omega)$.
For $i\ne k$ the
inclusion just obtained is impossible. This contradiction proves the
lemma.
\end{proof}

\subsection{Openness in the coefficientwise product topology}
\label{subsec:e8-product-openness}

Via the coefficient expansions in
Subsection~\ref{subsec:e8-setup}, identify the set of complex points
of the centered arc scheme with a subset
\[
\mathcal A(\mathbb C)\hookrightarrow
\prod_{m=1}^{\infty}\mathbb C^3,
\qquad
\alpha\longmapsto
\bigl((x_m(\alpha),y_m(\alpha),z_m(\alpha))\bigr)_{m\geq1}.
\]
The \emph{coefficientwise product topology} on
\(\mathcal A(\mathbb C)\) is the subspace topology induced by the
product of the usual Euclidean topologies on the factors
\(\mathbb C^3\). Equivalently, a basic neighborhood imposes open
conditions on only finitely many coefficients and imposes no
condition on the remaining coefficients. We distinguish the Zariski
topology on the scheme $\mathcal A$ from the coefficientwise product
topology on $\mathcal A(\mathbb C)$.
The topology is specified in each assertion, and the superscript
$\mathrm{prod}$ denotes closure in the latter topology.

\begin{theorem}
\label{thm:e8-product-openness}
For each $i=1,\ldots,8$, the exact-order stratum $W_i$ is a
Zariski-open subscheme of $\mathcal A$.
Consequently, $W_i(\mathbb C)$ is open in $\mathcal A(\mathbb C)$
for the coefficientwise product topology.
Moreover, for every $i\ne j$,
\[
W_i(\mathbb C)\cap C_j(\mathbb C)=\varnothing,
\qquad
W_i(\mathbb C)\cap
\overline{C_j(\mathbb C)}^{\,\mathrm{prod}}=\varnothing.
\]
\end{theorem}

\begin{proof}
We first establish the covering at the level of all scheme points.
Let $a\in|\mathcal A|$, put $K=\kappa(a)$, and consider the
corresponding arc
\[
\alpha_a:\operatorname{Spec}K[[t]]\longrightarrow X.
\]
Suppose that $a\ne0_\infty$.
At least one coordinate series of $\alpha_a$ is nonzero, so the
generic point of $\operatorname{Spec}K[[t]]$ maps to
$X\setminus\{0\}$.
Since $\pi:Y\to X$ is an isomorphism over this open subset,
the restriction to $\operatorname{Spec}K((t))$ has a unique lift
to $Y$. Properness of $\pi$ and the valuative criterion
\cite[Tag~0BX4]{Stacks} extend this lift uniquely to
\[
\widetilde{\alpha}_a:
\operatorname{Spec}K[[t]]\longrightarrow Y.
\]
Its center lies in the exceptional fiber, whose support is
$\bigcup_{j=1}^{8}E_j$.
Choose an $E_j$ containing that center; if the center lies on an
intersection of exceptional curves, either component may be chosen.
The point of $\operatorname{Arc}(Y)$ represented by
$\widetilde{\alpha}_a$ belongs to $\mathcal H_j$.
The inclusion $\pi_{\mathrm{arc}}(\mathcal H_j)\subset C_j$
established in Lemma~\ref{lem:e8-generic-orders} therefore gives
$a\in C_j$. This argument does not require $K$ to be algebraically
closed.

For $a=0_\infty$, choose a point of $E_j(\mathbb C)$ and take the
constant arc on $Y$ at that point. Its image is $0_\infty$, so
the same inclusion gives $0_\infty\in C_j$ for every $j$.
Since each $C_j$ is contained in $\mathcal A$, we obtain the
following equality of underlying topological spaces:
\begin{equation}\label{eq:e8-cover}
|\mathcal A|=\bigcup_{j=1}^{8}|C_j|.
\end{equation}

We next prove, first as an equality of underlying subsets of
$\mathcal A$, that
\begin{equation}\label{eq:e8-open-set}
W_i=D_{\mathcal A}(h_i)\setminus\bigcup_{j\ne i}C_j.
\end{equation}
If $a\in W_i$, then $h_i(a)\ne0$ by definition, and
Lemma~\ref{lem:e8-separation} gives $a\notin C_j$ for every
$j\ne i$. This proves one inclusion.
Conversely, suppose that $a$ belongs to the right-hand side of
\eqref{eq:e8-open-set}.
The covering \eqref{eq:e8-cover} places $a$ in some $C_k$, and
the defining exclusions force $k=i$.
By Lemma~\ref{lem:e8-generic-orders}, $C_i\subset Z_i$.
Thus $a\in Z_i\cap D_{\mathcal A}(h_i)=W_i$, proving the reverse
inclusion.

The right-hand side of \eqref{eq:e8-open-set} is Zariski open
in $\mathcal A$, because $D_{\mathcal A}(h_i)$ is open and the
union of the finitely many closed subsets $C_j$, $j\ne i$, is
closed. Since $\mathcal A$ is reduced, its open subscheme on this
subset is reduced. By definition, $W_i$ carries the induced
reduced structure on the same locally closed subset, so $W_i$
is precisely that open subscheme of $\mathcal A$.

It remains to verify the assertions about the coefficientwise
product topology. The function $h_i$ is a polynomial in finitely
many coefficient functions, so
$D_{\mathcal A}(h_i)(\mathbb C)$ is product-open.
Each $C_j$ is closed in the affine scheme $\mathcal A$ and hence
is defined by polynomial equations in the arc coefficients.
Although there may be infinitely many such equations, each involves
only finitely many coefficients and defines a product-closed zero
locus. Their intersection $C_j(\mathbb C)$ is therefore
product-closed.
Taking complex points in \eqref{eq:e8-open-set} expresses
$W_i(\mathbb C)$ as a finite intersection of product-open sets.
Finally,
\[
\overline{C_j(\mathbb C)}^{\,\mathrm{prod}}=C_j(\mathbb C),
\]
and both separation assertions follow from
Lemma~\ref{lem:e8-separation}.
\end{proof}

\begin{proposition}[Identification with the families of Example~14]
\label{prop:e8-arcology-correspondence}
Over $\mathbb C$, number the eight arc families in the order in which
they are described in Johnson--Koll\'ar's Example~14. The $i$-th
family is exactly $W_i(\mathbb C)$. The correspondence is
\[
\begin{array}{c|c|c}
\text{index here}&\text{construction in Example~14}&\ord(x,y,z)\\ \midrule
1&2\ord x=3\ord y=5\ord z=30&(15,10,6)\\
2,3,4,5&uv=w^5,\quad m=1,2,3,4&(3m,2m,m+1)\\
6,7&uv=w^3,\quad m=1,2&(5m,3m+1,2m)\\
8&uv=w^2,\quad m=1&(8,5,3).
\end{array}
\]
In the second row $i=m+1$, and in the third row $i=m+5$.
\end{proposition}

\begin{proof}
Johnson and Koll\'ar use
$\operatorname{Arc}^*(X_5)=\operatorname{Arc}(X_5,0)$
\cite[p.~523]{JK}. Hence their set of complex formal arcs is the same
as $\mathcal A(\mathbb C)$; the superscript $*$ means that the arcs
are centered at the origin and does not remove the constant zero arc.
Example~14 does not attach the labels $W_1,\ldots,W_8$ to the families
individually, but lists the eight exact order vectors in the order
shown above \cite[Example~14, pp.~530--531]{JK}. By definition,
\[
W_i(\mathbb C)=
\{\alpha\in\mathcal A(\mathbb C):\ord(x,y,z)=v_i\}.
\]
Indeed, $Z_i$ makes every coefficient below
$p_i,q_i,r_i$ vanish, while $h_i\ne0$ says that all three first
possible coefficients are nonzero. The first family in the table is
therefore $W_1(\mathbb C)$. Explicitly, each of its arcs has a unique
expression
\[
x=t^{15}a(t),\qquad y=t^{10}b(t),\qquad z=t^6c(t),
\]
with $a,b,c\in\mathbb C[[t]]^\times$ and $a^2+b^3=c^5$; conversely,
every triple of units satisfying this identity defines an arc in
$W_1(\mathbb C)$.

We now prove equality for each of the other seven parameterized
families. For $uv=w^n$, the $m$-th general family in \cite{JK} has
\[
\ord(u,v,w)=(m,n-m,1),\qquad1\leq m<n
\]
\cite[Example~9, p.~525]{JK}. For $n=5$, use the parametrization
\[
(x,y,z)=(u^3,u^2\lambda,uw),
\qquad uv=w^5,\qquad \lambda^3=v-1.
\]
Since $\ord v=5-m>0$, the series $\lambda$ is a unit, and
\[
\ord(x,y,z)=(3m,2m,m+1)\qquad(1\leq m\leq4).
\]
Thus the parameterized family is contained in
$W_{m+1}(\mathbb C)$. Conversely, given
$\alpha=(x,y,z)\in W_{m+1}(\mathbb C)$, choose
$u\in\mathbb C[[t]]$ with $u^3=x$ and set
\[
w=z/u,\qquad v=w^5/u,\qquad \lambda=y/u^2.
\]
These quotients belong to $\mathbb C[[t]]$ and satisfy
\[
\ord(u,v,w)=(m,5-m,1),\qquad uv=w^5.
\]
The equation $x^2+y^3=z^5$ gives $\lambda^3=v-1$, so every arc in
$W_{m+1}(\mathbb C)$ arises from the displayed parametrization.

For $n=3$, take
\[
(x,y,z)=(u^5,u^3w,u^2\mu),
\qquad uv=w^3,\qquad \mu^5=1+v.
\]
For $m=1,2$ this gives
\[
\ord(x,y,z)=(5m,3m+1,2m),
\]
so its image is contained in $W_{m+5}(\mathbb C)$. Conversely, for
$\alpha=(x,y,z)\in W_{m+5}(\mathbb C)$ choose $u^5=x$ and put
\[
w=y/u^3,\qquad v=w^3/u,\qquad \mu=z/u^2.
\]
The resulting series satisfy
\[
\ord(u,v,w)=(m,3-m,1),\qquad
uv=w^3,\qquad \mu^5=1+v,
\]
which proves the reverse containment.

Finally, for $n=2$ and $m=1$, the parametrization may be written as
\[
(x,y,z)=(u^7w,u^5,u^3\nu),
\qquad uv=w^2,\qquad \nu^5=1+v.
\]
It has order vector $(8,5,3)$, and hence its image is contained in
$W_8(\mathbb C)$. Given any $(x,y,z)\in W_8(\mathbb C)$, choose
$u^5=y$ and set
\[
w=x/u^7,\qquad v=w^2/u,\qquad \nu=z/u^3.
\]
Then $\ord(u,v,w)=(1,1,1)$, and the original equation gives
$uv=w^2$ and $\nu^5=1+v$. This proves equality with
$W_8(\mathbb C)$.

The formal roots used here exist for the following elementary reason.
If $d\geq1$, $q\geq0$, and
$f=t^{dq}a(t)\in\mathbb C[[t]]$ with $a(0)\ne0$, then, after choosing
a $d$-th root of $a(0)$, the coefficients of a unit series $b(t)$
satisfying $b(t)^d=a(t)$ are determined recursively and uniquely.
Thus $t^qb(t)$ is a $d$-th root of $f$. In the constructions above,
the relevant orders are divisible by $3$ or $5$ as required. The
inverse constructions therefore apply to every arc in the indicated
$W_i(\mathbb C)$, not only to a dense subfamily. Since the eight order
vectors are pairwise distinct, the eight families of Example~14 are
in one-to-one correspondence with
$W_1(\mathbb C),\ldots,W_8(\mathbb C)$.
\end{proof}

\begin{remark}[Topologies and external results]
Cylinders, generic points, irreducible components, and curve selection
in this section are used in the scheme-theoretic Zariski topology of
the arc space. Theorem~\ref{thm:e8-product-openness} first establishes
Zariski openness of the exact-order strata in $\mathcal A$ and then
deduces openness of their complex point sets in the coefficientwise
product topology. The separation lemma uses Reguera's
corrected curve selection lemma, Fern\'andez de Bobadilla's formal
wedge-adjacency theorem, and the Nash theorem for surfaces of
Fern\'andez de Bobadilla and Pe Pereira.
\end{remark}
\subsection{Proof of Theorem \ref{mtb}}

\begin{proof}
Theorem~\ref{thm:e8-product-openness} establishes the openness of
each $W_i(\mathbb{C})$
in $\mathcal{A}(\mathbb{C})$ for the coefficientwise product
topology, as well as the asserted separation from the product
closures of $C_j(\mathbb{C})$ for $j\ne i$.
The equality
$W_i(\mathbb{C})\cap C_j(\mathbb{C})=\varnothing$
also follows directly from Lemma~\ref{lem:e8-separation}.
Finally, Proposition~\ref{prop:e8-arcology-correspondence}
identifies these exact-order families
with the eight families of Johnson--Koll\'ar's Example~14,
in the order specified in Subsection~\ref{subsec:e8-setup}.
This proves Theorem~\ref{mtb}.
\end{proof}

%\section{Acknowledgement}
	
%	On behalf of all authors, the corresponding author states that there is no conflict of interest.  H. Zuo acknowledges support from NSFC (grant No. 12671056) and BJNSF (grant No. 1252009).

\section*{Acknowledgment of AI assistance}
During the preparation of this manuscript, the authors used GPT (OpenAI)
for language editing, assistance in checking computations in
Section~\ref{sec:solution}, and assistance in developing and checking parts
of the mathematical arguments in Section~\ref{sec:e8-open-families}.
The authors take full responsibility for all mathematical results and the final
content of this manuscript.

\end{document}